\documentclass[10pt]{article}

\usepackage[utf8]{inputenc}
\usepackage{fullpage}

\RequirePackage{verbatim}
\RequirePackage{hyperref}
\RequirePackage{amsmath,amsfonts,amssymb}

\usepackage{latexsym,amscd,amsthm}
\usepackage[all]{xy}
\usepackage{amstext,amsopn}
\usepackage{graphicx}
\usepackage{amscd}
\usepackage{url}
\usepackage{manfnt}
\usepackage{xcolor}
\usepackage{supertabular}

\usepackage{tikz-cd}
\usepackage{ulem} % to strike text

\usepackage{xcolor}
\usepackage{hyperref}
\hypersetup{
    colorlinks=true,
    linkcolor=[RGB]{13 71 161},		% Couleur des liens internes
    citecolor=[RGB]{13 71 161},		% Couleur des numéros de la biblio dans le corps
    urlcolor=[RGB]{13 71 161}		% Couleur des url
}
\usepackage[capitalize,nameinlink]{cleveref}
\Crefname{Thm}{Theorem}{Theorems}
\Crefname{ThmA}{Theorem}{Theorems}
\Crefname{Exa}{Example}{Examples}
\Crefname{Def}{Definition}{Definitions}
\Crefname{square}{Square}{Squares}
\Crefname{diagram}{Diagram}{Diagrams}

\newtheorem{Thm}{Theorem}[section]
\newtheorem{Prop}[Thm]{Proposition}
\newtheorem{Lem}[Thm]{Lemma}
\newtheorem{Cor}[Thm]{Corollary}

\newtheorem{ThmA}{Theorem}
\theoremstyle{remark}
\newtheorem{Rem}[Thm]{Remark}

\theoremstyle{definition}

\newtheorem{Def}[Thm]{Definition}
\newtheorem{Exa}[Thm]{Example}

\definecolor{mat}{HTML}{FFD6AD}
\definecolor{damien}{HTML}{ffadad}

\newcommand\coT{T^*}
\newcommand\coN{N^*}

\newcommand\g{\mathfrak{g}}

\newcommand\B[1]{B#1}

\newcommand\DR[1]{\mathrm{DR}^\bullet\!\left(#1\right)}
\newcommand\DRp[2]{\mathrm{DR}^{#1}\!\left(#2\right)}

\newcommand\Ap[2]{\mathcal{A}^{#1}\!\left(#2\right)}
\newcommand\Apcl[2]{\mathcal{A}^{#1,cl}\!\left(#2\right)}
\newcommand\Sym[2]{\mathrm{Sym}^\bullet_{#1}\left(#2\right)}

\newcommand\Lag[1]{\mathrm{Lag}\left(#1\right)}
\newcommand\bfLag{\mathbf{Lag}}

\newcommand\Crit{Crit}
\newcommand\Graph{Graph}

\newcommand\cO{\mathcal O}

\newcommand\pt{pt}
\newcommand\Id{\mathrm{Id}}

\newcommand\red{_{red}}

\newcommand\Map[3]{\mathrm{Map}_{#1}\left(#2,#3\right)}

\newcommand\Tr[1]{\mathrm{Tr}\left(#1\right)}

\newcommand\bfone{\mathbf{1}}

\author{
Mathieu Anel\footnote{Department of Philosophy, Carnegie Mellon University, mathieu.anel@protonmail.com}
\and
Damien Calaque\footnote{IMAG, Univ. Montpellier, CNRS, Montpellier, France, damien.calaque@umontpellier.fr}
}
\title{Shifted symplectic reduction \\ of derived critical loci}
\date{}

\begin{document}

\maketitle

\begin{abstract}
We prove that the derived critical locus of a $G$-invariant function $S:X\to\mathbb{A}^1$ carries a shifted moment map, and 
that its derived symplectic reduction is the derived critical locus of the induced function $S\red: X/G\to\mathbb{A}^1$ on the orbit stack. 
We also provide a relative version of this result, and show that derived symplectic reduction commutes with derived lagrangian intersections. 
\end{abstract}

\setcounter{tocdepth}{2}
\tableofcontents

%%%%%%%%%%%%% INTRODUCTION %%%%%%%%%%%%%%%%

\section*{Introduction}
\addcontentsline{toc}{section}{Introduction}

This paper is concerned with the derived symplectic geometry (in the sense of \cite{PTVV}) of critical loci in the presence of symmetries. 
Derived symplectic geometry can be seen as a model independent, homotopy invariant, and global way of dealing with the $QP$-manifolds of \cite{AKSZ}. 

\subsection*{Motivation: the BV formalism}

The seminal work \cite{AKSZ} takes its roots in the so-called \textit{Batalin--Vilkovisky (BV) formalism} \cite{BV}. 
In gauge theories, path integrals are ill-defined because of the presence of infinite dimensional symmetries. 
A way to make sense of them is to remove the redundant variables/fields by fixing the gauge, which of course destroys 
gauge invariance. 
A first attempt to face this problem is to introduce new fields, called \textit{ghosts}, which mathematicians know very well 
as Chevalley generators (for infinitesimal symmetries of the lagrangian). 
A systematic treatment has been initiated in \cite{BV} and is referred to as the BV formalism. 
Mathematicians have shown a great interest in understanding the deep nature of the BV formalism and its precursor, the BRST formalism. 
They have in particular given attention to two easier variants: 
\begin{itemize}
\item Finite dimensional (as opposed to the infinite dimensional situation physicists are interested in\footnote{Many infinite dimensional complications in physics come from the fact that the stacks involved naturally appear as quotients of infinite dimensional groupoids, 
but they can often be presented also by means of finite dimensional groupoids. }). 
\item Classical (as opposed to quantum). 
\end{itemize}
The BRST formalism has quickly been interpreted in terms of Hamiltonian reduction by Kostant--Sternberg \cite{KS}, of which we know 
from \cite{Cal15} that it can be understood as a derived lagrangian intersection. 
This interpretation has been extended to the realm of derived Poisson geometry by Safronov \cite{Safred}: Poisson reduction 
is indeed a derived coisotropic intersection. 
A systematic mathematical treatment of the classical finite dimensional BV construction has been initiated by Felder--Kazhdan in 
\cite{FK}, using pre-derived geometric tools. 
Apart from ghosts (and, possibly, ghosts for ghosts), the BV formalism also introduces other fields called \textit{anti-fields}. 
Very roughly, they are here to cure the possible presence of degenerate or non-isolated critical points of the action functional $S$. 
Let us give a quick description of the BV procedure, whose goal is to construct a certain differential graded algebraic object: 
\begin{itemize}
\item For any field variable $x_i$, add a dual anti-field variable $\xi_i=\partial_i$ of cohomological degree $-1$. 
The differential of $\xi_i$ is the $x_i$-derivative $\partial_iS$ of the action functional. 
\item Add degree $-2$ variables for every generator of the infinitesimal symmetries of $S$. 
They are usually called \textit{anti-ghost} variables and are sent, by the differential, to the corresponding vector field 
(a linear combination of the $\xi_i$'s which coefficients are functions of the $x_i$'s) killing the action functional $S$. 
\item To every anti-ghost variable, add a dual ghost variable of degree $1$ (a Chevalley generator). 
\end{itemize}
Usually, the process goes on because there might be higher symmetries, but we will stop here to simplify the exposition. 
There is a perfect duality (between field and anti-field variables, and between ghost and anti-ghost variables) that has a cohomological shift. 
This is the manifestation of the presence of a $(-1)$-shifted symplectic structure (in the sense of \cite{PTVV}). 

\medskip

The appearance of anti-fields is well-understood. 
The space of critical points of the action functional is the intersection of the graph of $d_{dR}S$ with the zero section in a cotangent space. 
Classically, one may need to perform a small geometric perturbation in case the intersection is not transversal. 
Anti-field variables $\xi_i$'s are here to perform a \textit{homological perturbation} instead, computing the \textit{derived} intersection of the zero section 
with the graph of $d_{dR}S$. 
Both are lagrangian, so that their derived intersection, which is known as the \textit{derived critical locus} of $S$, 
is $(-1)$-shifted symplectic (see \cite{PTVV,Ve}). 

\begin{quote}
\textbf{The main goal of this paper is to suggest an accurate interpretation 
of the classical BV formalism in the setting of derived symplectic geometry. }
\end{quote}

\subsection*{Main results}

At the heart of the BV--BRST formalism is the wish to compute (possibly ill-defined) path integrals by taking gauge symmetries out (e.g.~in order to get rid of infinite dimensional issues). 
Perturbatively, such an integral localizes at the critical points of the action functional $S:X\to\mathbb{A}^1$, and the BV formalism takes care of the possible 
degenerate nature of the critical locus by considering the \textit{derived critical locus} instead.
The derived critical locus is defined as the lagrangian intersection of the zero section with the graph of $d_{dR}S$ in $\coT X$, it is therefore $(-1)$-shifted symplectic \cite{PTVV,Ve}.

%\begin{itemize}
%\item At the heart of the BV--BRST formalism is the wish to compute (possibly ill-defined) path integrals by taking gauge 
%symmetries out (e.g.~in order to get rid of infinite dimensional issues). 
%\item Perturbatively, the integral localizes at critical points of the action functional $S:X\to\mathbb{A}^1$, and the BV 
%formalism takes care of the possible degenerate \notemat{singular?} nature of critical points by considering the \textit{derived critical locus} instead. 
%\item Being a lagrangian intersection (of the zero section with the graph of $dS$ in $\coT X$), the derived critical locus is 
%$(-1)$-shifted symplectic \cite{PTVV,Ve}.
%\end{itemize}

\bigskip

The symmetries of the action functional $S$ are in general a groupoid and not a group, but we will restrict ourselves to the latter situation for the clarity of exposition.
Notice, though, that what many people call \textit{the} BV formalism is the case when one considers the maximal homotopy (formal) groupoid of (higher) symmetries of $S$. 
This will be the matter of further investigations in subsequent work.

We shall here focus on the following situation: $X$ will be an algebraic variety (or stack), $G$ an algebraic group acting on $X$, and $S:X\to \mathbb A^1$ will be an invariant function, for which we denote by $S\red:X/G \to \mathbb A^1$ the induced function on the quotient stack $ X/G$.
From there, the BV procedure is usually done in different order by mathematicians or physicists:

\begin{center}
\begin{tabular}{l|l}
\textbf{In math references} (e.g.~\cite{FK,Pau}) & \textbf{In the original physics litterature} 
\\ 
\hline 
\\[-0.3cm]
field variables & field variables \\
anti-field variables (Koszul) & ghost variables (symmetries) \\
anti-ghost variables (Tate) & anti-field variables \\
ghost variables (Chevalley) & anti-ghost variables 
\end{tabular}
\end{center}
In the left column, anti-field and anti-ghost variables appear while taking a Koszul--Tate resolution 
(see \cite[Proposition 3.1.1]{Pau} and \cite[\S4.1]{FK}): more precisely, according to \cite{Sta} anti-field variables are 
identified as Koszul generators and anti-ghost variables (i.e.~infinitesimal symmetries) as Tate generators. 
It is only at the last step that one introduces ghost variables (or Chevalley generators), which means that we are considering a 
stacky quotient by the gauge symmetries. 

The (shifted) symplectic geometric interpretation of this procedure is easier for the right column:
the appearance of ghost variables first means that we are first taking the stacky quotient by infinitesimal symmetries, and then computing the derived critical locus of the function $S\red$ on the quotient by adding anti-fields and anti-ghosts variables.
The interpretation of the left column is less obvious but revealed by the following fact, which is one of our main results (see also \cite{BSS}):

\begin{ThmA}[see \Cref{thm-main}]\label{thm:A}
The derived stack $\Crit(S)$ comes equipped with a $(-1)$-shifted moment map $\Crit(S)\to \mathfrak g^*[-1]$ such 
that its derived symplectic reduction $\Crit(S)\red$ is equivalent, as a $(-1)$-shifted symplectic stack, to the derived 
critical locus $\Crit(S\red)$ of $S\red$. 
\end{ThmA}

This gives the following interpretation of the left column in our table: anti-ghost variables appear when one takes the 
(derived) zeroes of the moment map (first step of symplectic reduction) and ghost variables appear as usual when one takes 
the quotient by $G$ (second step of symplectic reduction).  
In other words, the difference between the two columns is the permutation of the two steps of reduction and extremalization!
Strictly speaking, it is when one takes the quotient by the Lie algebra symmetries that we get ghost variables. The generalization of the classical BV formalism 
in finite dimensions from Lie algebra to group actions has been achieved in the recent work \cite{BSS}: the authors identify 
$\Crit(S\red)$ with a shifted symplectic reduction of $\Crit(S)$, they compute explicitely its function algebra, and they compare this algebra with the algebra 
produced by the classical BV formalism (also identifying $(-1)$-shifted symplectic structures on both sides). 

\bigskip

Our second main result is a useful variation of \Cref{thm:A} where we consider {\it constrained} derived critical loci (see \cite{RS}).
If $Y$ is another stack with a $G$-action and if $p:X\to Y$ is a $G$-equivariant map, we write 
$p\red: X/G\to Y/G$ for the induced map on quotient stacks. 
The constrained derived critical locus $\Crit_p(S)$ is defined as the derived intersection 
of the graph of $d_{dR}S$ with the conormal stack to the graph of $p$. 
This intersection comes with a natural map $\Crit_p(S)\to \coT Y$ and we can prove that it carries a canonical lagrangian structure. 
\begin{ThmA}[see \Cref{thm-main2}]\label{thm:B}
The lagrangian morphism 
\[\Crit_{p\red}(S\red)\longrightarrow \coT\big(Y/G\big)
\]
is a derived symplectic reduction of the lagrangian morphism 
\[
\Crit_p(S)\longrightarrow \coT Y
\]
along the natural moment map $\coT Y\to\mathfrak g^*$. 
\end{ThmA}
This result provides a nice derived geometric interpretation (still in the case of group of symmetries) of the so-called (classical) 
BV-BFV (for Batalin--Fradkin--Vilkovisky) formalism \cite{BFV-BV,BFV-BF,CMR1,CMR2,CM}.

\bigskip

The last of our main results is a conceptual explanation for \Cref{thm:A} which can be summarized by saying that 
{\it lagrangian intersection commutes with derived symplectic reduction}. 
Let $X$ be a symplectic stack equipped with a hamiltonian $G$-action, let $\mu_{(0)}:X\to\g^*$, be the moment map and $X\red$ be the derived symplectic reduction.
Let also $L$ and $L'$ be two lagrangians of $X$, and let $L\red$ and $L'\red$ be their derived symplectic reductions. Then, the following holds:

\begin{ThmA}[see \Cref{thm-main3}]\label{thm:C}
The moment map of $X$ induces a canonical $(-1)$-shifted moment map 
\[
\mu_{(-1)}:L\underset{X}{\times}L'\to\g^*[-1]
\]
and a natural equivalence 
\[
\big(L\underset{X}{\times}L'\big)\red\simeq L\red\underset{X\red}{\times} L'\red
\]
of $(-1)$-shifted derived stacks. 
\end{ThmA}

\subsection*{Description of the paper}

\Cref{section1} recalls basic facts of derived symplectic geometry (forms and closed forms on derived stacks, 
shifted symplectic structures, lagrangian structures and lagrangian correspondences, zeroes of closed $1$-forms, etc...). 
We emphasize the role of push and pull operations for lagrangians in shifted cotangent stacks, denoted $f_\dagger$ and $f^\dagger$, 
that are associated with a morphism $f:X\to Y$ of derived stacks. 

\medskip

\Cref{section2} proves a general statement (\Cref{thm:magic-cube}) about certain iterated intersections 
of lagrangian correspondences. 
It is inspired by an earlier result of Ben-Bassat \cite{BB}, and provides a convenient general framework that synthetizes the 
lagrangian intersection interpretation of the hamiltonian and quasi-hamiltonian formalisms \cite{Cal15,SafCS}. 

\medskip

\Cref{section3} states and proves \Cref{thm:A,thm:B} in the more general setting of (relative) zeroes of $1$-forms rather than just for (relative) derived critical loci. 
These results are direct consequences of the results 
in the previous section together with a Beck--Chevalley type property for the push and pull operations along a cartesian square of 
derived stacks (\Cref{prop-BC}). 

\medskip

\Cref{section3bis} proves \Cref{thm:C}, and explains how \Cref{thm:A} can be deduced from it. 

\medskip

\Cref{section4} is devoted to examples: twisted cotangent bundles of global quotient stacks, 
Hilbert schemes of points in $\mathbb{C}^3$, Chern--Simons theory (moduli of flat connections on $3$-folds), 
and Souriau's approach to Einstein's covariance principle.  

\subsection*{Acknowledgements}
The first author gratefully acknowledges the support of the Air Force Office of Scientific Research through grant FA9550-20-1-0305. 
The second author has received funding from the European Research Council (ERC) under the European Union's Horizon 2020 
research and innovation programme (grant agreement No 768679).

%%%%%%%%%%%%% SECTION 1 %%%%%%%%%%%%%%%%

\section{Derived symplectic geometry}\label{section1}

\subsection{The de~Rham complex, forms, and closed forms}

This paper is written using the language of $(\infty,1)$-category theory \cite{HTT} and derived algebraic geometry \cite{Lurie:SAG, TV}.
This means that, unless otherwise specified, all functors are implicitly assumed to be derived (limits and colimits are always homotopical, global sections are always derived...) and all objects are implicitely assumed to be derived stacks.
In particular, for a group $G$ acting on a stack $X$, $X/G$ refers always to the quotient stack.

We shall use the cohomological convention for the degree of (cochain) complexes. 
In particular, if $M$ is a complex $M[-1]$ denote the looping of $M$ and $M[1]$ its suspension.
For a graded complex, we shall use the term {\it degree} to refer to the cohomological degree, and the term \textit{weight} for the auxiliary grading. 
We shall identify the latter with a $~^\bullet$ symbol. 
A \textit{graded mixed complex} is the data of a graded complex together with a \textit{mixed differential}, which is a degree $1$ and weight $1$ 
endomorphism that squares to zero. 
For a graded mixed complex $C$, we denote by $C^\sharp$ its underlying graded complex. 
We refer to \cite{CPTVV} for details about the homotopy theory of graded mixed complexes. 

Recall from \cite{PTVV} that any derived stack $X$ has a graded mixed de~Rham complex $\DR X$ (beware that we shall rather follow the weight convention of \cite{CPTVV}). 
Whenever $X=Spec(A)$ is a derived affine scheme, with $A$ a cofibrant connective commutative differential graded algebra (cdga, for short), then 
\[
\DR X ^\sharp:=\Sym A {\Omega^1_A[-1]}
\]
and the mixed differential is given by the universal derivation $A\to \Omega^1_A$, that we extend \textit{via} the Leibniz rule.
More generally, if $X$ is an arbitrary stack, we define
\[
\DR X :=\underset{Spec(A)\to X}{\mathrm{lim}} DR\big(Spec(A)\big)\,.
\]
If $X$ is Artin (which, in this paper, means geometric and locally of finite presentation) 
and if $\mathbb L_X$ is its cotangent complex, we have
\[
\DR X ^\sharp \simeq \Gamma\big(\Sym {\cO_X} {\mathbb{L}_X[-1]} \big)\,.
\]

\begin{Def}
\begin{enumerate}
\item[(a)] A {\it $p$-form of degree $n$} on $X$ is an $n$-cocycle in $\DRp p X [-p]$. 
The space $\Ap p {X,n}$ of $p$-form of degree $n$ on $X$ is defined as (the realization of the connective complex) $\tau_{\leq 0}(\DRp p X [n-p])$.

\item[(b)] A {\it closed $p$-form of degree $n$} on $X$ is an $n$-cocycle in $\prod_{i\geq p}\DRp i X [-p]$ for the total differential (i.e.~the sum of 
the differential and the mixed differential).
The space $\Apcl p {X,n}$ of closed $p$-forms of degree $n$ on $X$ is defined as (the realization of) $\tau_{\leq 0}\left(\prod_{i\geq p}\DRp i X [n-p]\right)$.     
\end{enumerate}
\end{Def}

The functor $X\mapsto \DR X$ is a stack for the \'etale topology on connective cdga (see \cite{PTVV}).
It follows that the functors $X\mapsto \Ap p {X,n}$ and $X\mapsto \Apcl p {X,n}$ are also \'etale stacks, which we denote by $\Ap p n$ and $\Apcl p n$.

\begin{Exa}[$0$-forms]
The stack $\Ap 0 0$ of $0$-forms of degree $0$ is simply $\mathbb{A}^1$. 
More generally, we have $\Ap 0 n \simeq \mathbb{A}^1[n]$, that is $\Ap 0 {X,n} = \tau_{\leq 0}\left(\Gamma(O_X)[n]\right)$. 
\end{Exa}

\begin{Exa}[Smooth schemes]\label{example:smoothscheme}
Let $X$ be a smooth scheme of finite presentation. 
Then 
\[
\DR X ^\sharp\simeq \Gamma\big(\Sym {\cO_X} {\Omega^1_X[-1]}\big)\,.
\]
The mixed differential preserves the graded subcomplex
\[
\Omega^\bullet(X):=\Gamma_0\big(\Sym {\cO_X} {\Omega^1_X[-1]}\big)
\]
of \textit{underived} global sections of $\Sym {\cO_X} {\Omega^1_X[-1]}$ and acts as the ordinary de~Rham differential $d_{dR}$ on it. 
\end{Exa}

\begin{Exa}[Classifying stacks]\label{exa-classifying}
Let $G$ be an affine algebraic group, $\g$ its Lie algebra and $\B G$ its classifying stack. 
Then,
\[
\DR {\B G}^\sharp\simeq C\big(G,\Sym {} {\g^*[-2]}\big)\,,
\]
and the mixed differential vanishes on $\Sym {}{\g^*[-2]} ^G$ (for degree reasons). 
Observe that we have implicitly used the identification $\mathbf{QCoh}(\B G)\simeq G\mathbf{-mod}$ between the stable $(\infty,1)$-category 
of quasi-coherent sheaves on $BG$ with the $(\infty,1)$-category of complexes of $G$-modules. Through this identification, the (derived) global section 
functor corresponds to the group cochains functor $C(G,-)$. 

In particular, the space of closed $1$-forms of degree $1$ on $\B G$ is equivalent to the set $(\mathfrak g^*)^G$ of infinitesimal characters. 
If $G=\mathbb{G}_m$ there is (up to scaling) a canonical infinitesimal character.
This character $\B{\mathbb{G}_m}\to \Apcl 1 1$ actually comes from the group morphism 
\[
d_{dR}log\,:\,\mathbb{G}_m\longrightarrow \Apcl 1 0 = \Omega \Apcl 1 1
\]
classifying the (degree $0$) multiplicative closed $1$-form $d_{dR}log(z)=(d_{dR}z)/z$.
Recall that an infinitesimal character $\phi:\B G\to \Apcl 1 1$ is integrable if its looping $\Omega\phi: G\to \Apcl 1 0$ factors through $d_{dR}log$.
\end{Exa}

\begin{Exa}[Global quotients]\label{example:globalquotient}
Let $X$ be a smooth scheme of finite presentation with an action of an affine algebraic group $G$. 
If $\g$ is the Lie algebra of $G$, we have
\[
\DR {X/G} ^\sharp \simeq 
C\big(G,\Gamma\big(\Sym {\cO_X} { \Omega^1_X[-1]\to\cO_X\otimes\g^*[-2]} \big)\,,
\]
where the complex $\Omega^1_X[-1]\to\cO_X\otimes\g^*[-2]$ is in degree 1 and 2 and its differential is the transpose of the infinitesimal action 
$\g \to \Gamma(T_X)$, denoted $x\mapsto\vec{x}$.
The mixed differential preserves the graded subcomplex 
\[
\big(\Omega^\bullet(X)\otimes \Sym{}{\g^*[-2]}\big)^G
\]
of \textit{underived} global section, on which it acts as $d_{dR}\otimes{id}$. Recall that elements in the above complex can be viewed as $G$-equivariant $\Omega(X)$-valued polynomial functions on $\g$, 
and that the differential is given by $(d\alpha)(x)=\iota_{\vec{x}}(\alpha(x))$, where $x\in\g$ 
and $\iota$ denotes the contraction of a form by a vector field.
\end{Exa}

\subsection{Shifted symplectic and lagrangian structures}

Following \cite{Hau}, we consider the $(\infty,1)$-category $\mathbf{Span}(\mathbf{dSt}_{/\Apcl 2 n})$ of correspondences of derived stacks over $\Apcl 2 n$. 
Its objects are called \textit{$n$-shifted presymplectic stacks}, and its $1$-morphisms are called \textit{$n$-shifted isotropic correspondences}. 
We may sometimes omit to specify ``$n$-shifted'' when it is obvious from the context. This $(\infty,1)$-category carries a symmetric monoidal structure, 
induced by the cartesian product $\times$ of derived stacks and the addition map $+:\Apcl 2 n \times \Apcl 2 n \to \Apcl 2 n$. 
Every object is dualizable for this symmetric monoidal structure: the dual $\overline{X}$ of an object $X$ is given by composing with the opposite morphism 
$\Apcl 2 n\to \Apcl 2 n$; in other word, $\overline{X}$ is the same derived stack with the opposite closed $2$-form. 

An \textit{isotropic morphism} $X\to Y$ is an isotropic correspondence $\pt_{(n)}\leftarrow X\rightarrow Y$, where $\pt_{(n)}$ denotes the terminal 
stack equipped with its only $n$-shifted closed $2$-form (which is zero). 

\medskip

The notions of $n$-shifted presymplectic stacks and isotropic correspondences have a condition of non-degeneracy that we do not recall (see \cite{PTVV,Cal15}).
Non-degenerate objects are called \textit{$n$-shifted symplectic stacks} and non-degenerate 1-morphisms are called \textit{$n$-shifted lagrangian correspondences}.
Recall from \cite{Cal15} that the non-degeneracy condition is preserved by the composition and the monoidal product.
Hence $n$-shifted symplectic stacks and $n$-shifted lagrangian correspondences define a symmetric monoidal $(\infty,1)$-subcategory $\bfLag_n$ of 
$\mathbf{dSt}_{/\Apcl 2 n}$. 
We refer to \cite{Hau} for details. 

We define an ($n$-shifted) \textit{lagrangian morphism} $L\to X$ to be a lagrangian correspondence $\pt_{(n)}\leftarrow L\rightarrow X$ and we denote by 
$\mathrm{Lag_n}(X):=\Map {\bfLag_n} {\pt_{(n)}} X$ the space of lagrangian morphisms with codomain $X$.
When $n$ is understood from the context, we shall simply write $\Lag X$ or $\Map {\bfLag} Y X$ for the mapping spaces in $\bfLag_n$.

\medskip

The construction of correspondences can be iterated into correspondences of correspondences. 
This provides the symmetric monoidal $(\infty,2)$-category
$\mathbf{Span}^2(\mathbf{dSt}_{/\Apcl 2 n})$ of $2$-iterated correspondences of derived stacks over $\Apcl 2 n$ (see \cite{Hau}), as well as its symmetric 
monoidal $(\infty,2)$-subcategory $\mathbf{Lag}_n^{(2)}$ of $2$-iterated lagrangian ones (see \cite{CHS}, as well as the heuristic 
construction in \cite{Cal15}). 
We will actually only need the truncated version from \cite{ABB}, where the authors construct a symmetric monoidal bicategory of 
$2$-iterated lagrangian correspondences (which is sufficient for our purposes, and coincides with the homotopy bicategory of $\mathbf{Lag}_n^{(2)}$).  
We briefly recall how to describe the $2$-morphisms.
Given two objects (i.e.~$n$-shifted symplectic derived stacks) $X,Y$, and two $1$-morphisms $L,M$ from $X$ to $Y$, we can 
view these $1$-morphisms (i.e.~lagrangian correspondences) as lagrangian morphisms with codomain $\overline{X}\times Y$. 
Now, we can use that derived lagrangian intersections are symplectic (see \cite{PTVV}) to get an $(n-1)$-shifted symplectic structure 
on $L\underset{\overline{X}\times Y}{\times}M$. 
It turns out that $2$-morphisms $L\Rightarrow M$ are exactly lagrangian morphisms with codomain $L\underset{\overline{X}\times Y}{\times}M$. 

\subsection{Important examples}

\subsubsection{Shifted cotangent and conormal stacks}
\label{sssec-cotcon}

Let $X$ be a derived Artin stack. 
Then the shifted cotangent stack $\coT[n]X$ comes equipped with a tautological $1$-form $\lambda_X$ of degree $n$, which corresponds to the classifying 
map $id:\coT[n]X\to \coT[n]X$. 
Then the degree $n$ closed $2$-form $\omega_X:=d_{dR}\lambda_X$ is non-degenerate (and thus is an $n$-shifted symplectic structure). Moreover, the zero section $X\to\coT X$ is naturally equipped with a lagrangian structure. 
We refer to \cite{Cal17} for more details. 
\begin{Exa}
Recall that, for an algebraic group $G$, $\coT[n]\B G \simeq \g^*[n-1]/G$ (see \Cref{sec-momentmaps} for an explicit description of the shifted symplectic structure). 
\end{Exa}

\medskip

For every morphism $f:X\to Y$ of derived Artin stacks we have a commuting diagram 
\[
\begin{tikzcd}[sep=small]
& f^*\coT[n]Y \ar[dl] \ar[dr]	&\\
\coT[n]X \ar[rd,"\lambda_X"'] && \coT[n]Y \ar[ld,"\lambda_Y"]\,.	\\
& \mathcal{A}^1(n)
\end{tikzcd}
\]
\begin{Exa}
If $Y=\pt$, then $f^*\coT[n]Y\simeq X\to \coT[n]X$ is the zero section. If $X=\pt$, then $f=y$ is a point in $Y$, and 
$f^*\coT[n]Y\simeq T^*_yY[n]\to \coT[n]Y$ is the fiber at $y$ of the shifted cotangent. 
\end{Exa}
Composing with $d_{dR}:\Ap 1 n\to\Apcl 2 n$ thus produces an isotropic correspondence, which can be shown to be a lagrangian one (along the same lines as in \cite{Cal17}). This leads to two maps 
\begin{eqnarray*}
f^\dagger\,:\,\Lag {\coT[n]Y}											& \longrightarrow & \Lag {\coT[n]X} \\
\mathrm{and}\qquad f_\dagger\,:\,\Lag{\coT[n]X}	& \longrightarrow & \Lag{\coT[n]Y}\,.
\end{eqnarray*}
Associativity of the composition of lagrangian correspondences tells us that for two lagrangian morphisms $L_X\to \coT[n]X$ and $L_Y\to \coT[n]Y$, we have an equivalence of $(n-1)$-shifted symplectic derived stacks 
\begin{equation}
\label{eq-push-pull}
f_\dagger(L_X)\underset{\coT[n]Y}{\times}L_Y\simeq L_X\underset{\coT[n]X}{\times}f^\dagger(L_Y)\,.
\end{equation}
\begin{Rem}
Actually, $f^\dagger$ more generally defines a functor 
\[
\bfLag_{\coT[n]Y/}\longrightarrow\bfLag_{\coT[n]X/}
\]
and $f_\dagger$ defines a functor 
\[
\bfLag_{/\coT[n]X}\longrightarrow\bfLag_{/\coT[n]Y}\,.
\]
%Up to the canonical equivalence $\bfLag_{\coT[n]X/}\simeq \bfLag_{/\coT[n]X}$, these two functors are both left and right adjoint to each other. 
%\notemat{check this}\\
%\damienline{Attention : $C=C^{op}$ implique que $C_{/x}=(C_{x/})^{op}$, mais je doute qu'on puisse identifier $C_{/x}$ et $C_{x/}$. Ca n'est pas un probleme 
% (il suffit de prendre le ``op'' d'un des deux foncteurs). Par contre je ne suis pas certain de cette histoire d'adjonction entre categories (co)slice. }
Moreover, if we are given two lagrangian correspondences 
\[
\begin{tikzcd}[sep=small]
& L_X	\ar[rd]\ar[ld]\\
W 	&& \coT[n]X
\end{tikzcd}
\quad\mathrm{and}\quad
\begin{tikzcd}[sep=small]
& L_Y	\ar[rd]\ar[ld]\\
\coT[n]Y && Z
\end{tikzcd}
\]
then the equivalence \eqref{eq-push-pull} still holds (but in $\Map \bfLag W Z$ now).  
\end{Rem}
\begin{Exa}\label{example-conormal}
Observe that the underlying morphism of $f_\dagger(X\overset{0}{\to}\coT[n]X)$ is the shifted conormal  
\[
\coN[n]f\,:\,\coT_X[n]Y\longrightarrow \coT[n]Y\,,
\]
where $\coT_X[n]Y$ is the stack of sections of the relative tangent complex $\mathbb{L}_f[n-1]$, and $\coN[n]f$ is induced by the morphism 
$\mathbb{L}_f[n-1]\to f^*\mathbb{L}_Y$. Therefore, $\coN[n]f$ is equipped with a lagrangian structure (as shown in \cite{Cal17}). 
In particular, if we are given a sequence $X\overset{f}{\to}Y\overset{g}{\to}Z$ of derived Artin stacks, we get the following composition of 
lagrangian correspondences: 
\[
\begin{tikzcd}[sep=small]
&&& \coT_X[n]Z\ar[dl]\ar[dr]	&&&\\
&& \coT_X[n]Y\ar[dl]\ar[dr]	&& f^*g^*\coT[n]Z\ar[dl]\ar[dr] 	&&	\\
& X\ar[dr,"0"]\ar[dl]	&& f^*\coT[n]Y\ar[dl]\ar[dr]	&& g^*\coT[n]Z\ar[dl]\ar[dr]	&	\\\pt_{(n)}	&& \coT[n]X && \coT[n]Y && \coT[n]Z \,.
\end{tikzcd}
\]
\end{Exa}
\begin{Rem}
A more conceptual way of reformulating the above is to say that $\coT[n]$ defines an $(\infty,1)$-functor 
$\mathbf{dSt}^{Art}\longrightarrow \bfLag_n$, where $\mathbf{dSt}^{Art}$ is the full sub$(\infty,1)$-category of $\mathbf{dSt}$ 
spanned by derived Artin stacks. 
\end{Rem}
\begin{Rem}
We let the reader check that the following $n$-shifted symplectic derived stacks are all equivalent: 
\[
\coT[n]X\,,\quad X\underset{\coT[n+1]X}{\times}X\,,\quad\mathrm{and}\quad \coT_X[n+1]X\,.
\]
Here we implicitly make use of the identification between $n$-shifted symplectic derived stacks and lagrangian morphisms with codomain $\pt_{(n+1)}$ (see \cite{Cal15}). 
\end{Rem}

\subsubsection{Derived critical locus}

Let $X$ be a derived Artin stack, and let $\alpha$ be a closed $1$-form of degree $n$ on $X$. 
Then recall from \cite{Cal17} that the section $X\to \coT[n]X$ of the shifted cotangent stack given by its underlying $1$-form $\alpha_0$ of degree $n$ 
comes equipped with a lagrangian structure. We actually have a map 
\[
\Graph\,:\,\Apcl 1 {X,n} \longrightarrow \Lag{\coT[n]X}\,.
\]
\begin{Def}
We define the $(n-1)$-shifted symplectic derived stack $Z(\alpha)$ as the derived lagrangian intersection 
\[
Z(\alpha):=\Graph(\alpha)\underset{\coT[n]X}{\times}\Graph(0)=\Graph(\alpha)\underset{\coT[n]X}{\times}X
\]
of $\Graph(\alpha)$ with the zero section $\Graph(0)$. 
If $\alpha=d_{dR}f$, where $f$ is a degree $n$ function on $X$, then we write 
\[
\Crit(f):=Z(d_{dR}f)
\]
and call it the \textit{derived critical locus} of $f$. 
\end{Def}
\begin{Rem}
Observe that the zero section is nothing else than $\coN[n]q$ (see \cref{example-conormal}), where $q:X\to \pt$ is the terminal morphism. 
Hence $Z(\alpha)$ is by definition $q_\dagger \Graph(\alpha)$. 
\end{Rem}
\begin{Exa}[Twisted cotangent bundles]\label{ex-twisted}
Let $X$ be a smooth scheme, and consider its shifted cotangent stack $\coT[1]X$. Recall from \cite{PTVV} that: 
\begin{itemize}
\item $\pi_0\big(\Ap 1 {X,1}\big)\simeq H^1(X,\Omega^1_X)$. 
\item $\pi_0\big(\Apcl 1 {X,1}\big)\simeq H^1(X,\Omega^{1,cl}_X)$. 
\end{itemize}
Hence any element $\alpha\in H^1(X,\Omega^{1,cl}_X)$ provides us with a lagrangian morphism $X\to \coT[1]X$. 
Its derived intersection $Z(\alpha)$ with the zero section thus inherits a $0$-shifted symplectic structure. 
It can be identified with the twisted cotangent bundle $\coT_\alpha X$ (see \cite{Hab}). 
\end{Exa}
Let $f:X\to Y$ be a morphism of derived Artin stacks, and let $\alpha$ be a closed $1$-form of degree $n$ on $Y$. 
We have the following commuting diagram, where the square is cartesian: 
\[
\begin{tikzcd}[sep=small]
&& X \ar[dl]\ar[dr]\ar[ddll,"f^*\alpha_0"', bend right]\\
& f^*\coT[n]Y\ar[dl]\ar[dr]	&& Y \ar[dl,"\alpha_0"']\ar[dr]\\
\coT[n]X 	&& \coT[n]Y && \pt\,.
\end{tikzcd}
\]
\begin{Lem}\label{lem-graph-pullbacks}
The above diagram still holds over $\Apcl 2 n$. 
Therefore there is an equivalence $f^\dagger\big(\Graph(\alpha)\big)\simeq graph\big(f^*\alpha\big)$, in $\Lag{\coT[n]X}$.  
\end{Lem}
\begin{proof}
By definition of the tautological $1$-form on shifted cotangent stacks, and from the fact that $\alpha_0$ lifts to a closed form, 
we have the following commuting diagram
\[
\begin{tikzcd}[sep=small]
&& X \ar[dl]\ar[dr]\ar[ddll,"f^*\alpha_0"', bend right]\ar[ddrr,"f^*\alpha", bend left]\\
& f^*\coT[n]Y\ar[dl]\ar[dr]	&& Y \ar[dl,"\alpha_0"]\ar[dr,"\alpha"']	\\
    \coT[n]X\ar[rrd] 	& 									& \coT[n]Y \ar[d]																	& & \Apcl 1 n \,.\ar[lld] \\
							&& \Ap 1 n &&    
\end{tikzcd}
\]
Then the result follows by composing the morphism $\Apcl 1 n\to\Ap 1 n$ with  
\[
\begin{tikzcd}
	\Apcl 1 n \ar[r]\ar[d] & \pt \ar[d] 					\\
	\Ap 1 n \ar[r,"d_{dR}"] 	& \Apcl 2 n  \,.
\end{tikzcd}
\]
\end{proof}

%%%%%%%%%%%%% SECTION 2 %%%%%%%%%%%%%%%%

\section{The ``magic cube'' of derived symplectic reduction}\label{section2}

\subsection{The ``magic cube''}

We introduce a useful construction producing a lagrangian correspondence from three lagrangian correspondences sharing the same domain.
We consider $X_0$, $X_1$, $X_2$, and $X_3$, four $n$-shifted symplectic derived stacks, and three lagrangian correspondences $L_{0i}\in \Map \bfLag {X_0}{X_i}$. 
By convention, $L_{i0}\in \Map \bfLag {X_i}{X_0}$ is the opposite lagrangian correspondence. 
We consider the composed correspondences $L_{ij}:=L_{i0}\times_{X_0}L_{0j}\in \Map \bfLag {X_i}{X_j}$ and we define the cyclic derived lagrangian intersection
\begin{align*}
X_{123} & := (L_{10}\underset{X_0}{\times}L_{02}\underset{X_2}{\times}L_{20}\underset{X_0}{\times}L_{03}
					\underset{X_3}{\times}L_{30}\underset{X_0}{\times}L_{01})\underset{\overline{X_1}\times X_1}{\times}X_1\\
			&  = (L_{12}\underset{X_2}{\times}L_{23}\underset{X_3}{\times}L_{31})\underset{\overline{X_1}\times X_1}{\times}X_1 \,.
\end{align*}
In other words, the derived stack $X_{123}$ is defined as the limit of the hexagon of (lagrangian) correspondences
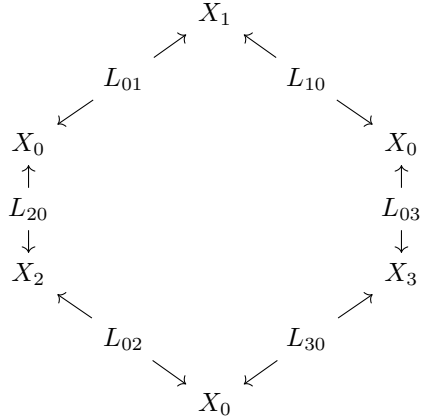
\[
\begin{tikzcd}[sep=small]
&& X_1 \\
&L_{01}\ar[ru]\ar[ld]&&L_{10}\ar[lu]\ar[rd]\\
X_0 && && X_0\\
L_{20} \ar[u]\ar[d]&&&& L_{03}\ar[u]\ar[d]\\
X_2 &&&& X_3 \,.\\
&L_{02}\ar[lu]\ar[rd]&&L_{30}\ar[ru]\ar[ld]\\
&& X_0
\end{tikzcd}
%\qquad
%\begin{tikzcd}[sep=small]
%&& X_1 \\
%&L_{01}\ar[ru]\ar[ld]&&L_{10}\ar[lu]\ar[rd]\\
%X_0 &&&& X_0\\
%L_{20} \ar[u]\ar[d]&&
%X_{123}
%\ar[uuu,dashed]
%\ar[luu,dashed]
%\ar[llu,dashed]
%\ar[ll,dashed]
%\ar[lld,dashed]
%\ar[ldd,dashed]
%\ar[ddd,dashed]
%\ar[rdd,dashed]
%\ar[rrd,dashed]
%\ar[rr,dashed]
%\ar[rru,dashed]
%\ar[ruu,dashed]
%&& L_{03}\ar[u]\ar[d]\\
%X_2 &&&& X_3\\
%&L_{02}\ar[lu]\ar[rd]&&L_{30}\ar[ru]\ar[ld]\\
%&& X_0
%\end{tikzcd}
\]

We define also $L_{123}:=L_{01}\underset{X_0}{\times}L_{02}\underset{X_0}{\times}L_{03}$.
This stack fits in a cube
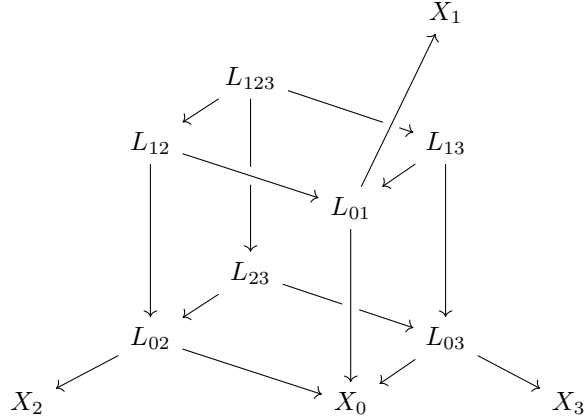
\[
\begin{tikzcd}[sep=small]
&&&&& X_1\\
&&& L_{123} \ar[ld]\ar[rrd]\ar[ddd] \\
&&L_{12} \ar[ddd]		&&& L_{13}\ar[ld]\ar[ddd]\\
&&&& L_{01} \ar[from=llu, crossing over]\ar[uuur, crossing over]\\
&&& L_{23} \ar[ld]\ar[rrd]\\
&&L_{02} \ar[rrd]\ar[lld] &&&	L_{03} \ar[ld]\ar[rrd]\\
X_2&&&& X_0 \ar[from=uuu,crossing over]&&& X_3
\end{tikzcd}
\]
(where we have draw extra legs to help see the three original correspondences).
Every face of the cube is cartesian and the object $L_{123}$ is the limit of the front three edges of the cube.
This is the \textit{magic cube} giving its name to the construction.
We shall see many examples.

\begin{Thm}
\label{thm:magic-cube}
The stack $X_{123}$ is $(n-1)$-shifted symplectic, and there exists a canonical morphism 
\[
L_{123}\longrightarrow X_{123}
\]
which is lagrangian.
\end{Thm}

\begin{Rem}
This theorem has two natural generalizations (which can be proven similarly).
The first one is that the result holds for $n$-presymplectic stacks and isotropic correspondences.
The second one is that it holds more generally starting with an arbitrary number of correspondences sharing the same domain $X_0 \leftarrow L_{0i}\to X_k$.
Then, the cyclic derived lagrangian intersection 
\[
X_{12\dots k}:=(L_{12}\underset{X_2}{\times}\cdots \underset{X_k}{\times}L_{k1})\underset{\overline{X_1}\times X_1}{\times}X_1
\]
is $(n-1)$-shifted symplectic, and the isotropic morphism 
\[
L_{12\dots k}:=L_{01}\underset{X_0}{\times}\cdots \underset{X_0}{\times}L_{0k}\longrightarrow X_{12\dots k}
\]
is lagrangian. 
\end{Rem}

The first part of the statement is a particular instance of \cite[Theorem 3.1]{BB} (with $S=\pt_{(n)}$). 
The proof is simple and boils down to the fact that a cyclic composition of lagrangian 
correspondences is shifted symplectic: the shifted closed $2$-form is given by a loop in the space of closed $2$-forms, and checking non-degeneracy is standard. 
The second part of the statement can be proved using a similar reasoning as in \textit{loc.~cit.}, but we shall provide a more conceptual proof.

First, we need some recollection about traces in 2-categories with duals.
Recall from \cite{ABB,CHS} the symmetric monoidal $(\infty,2)$-category $\mathbf{Lag}_n^{(2)}$ of $2$-iterated lagrangian correspondences between $n$-shifted symplectic derived stacks (see \cite{Hau} for the isotropic version). 
This category has the following features: 
\begin{itemize}
\item The monoidal structure is $\times$, and its unit is $\pt_{(n)}$. 
\item All objects are dualizable, and the dual of an $n$-shifted symplectic derived stack $X$ is $\overline{X}$. 
Evaluation and coevaluation are both, as the identity $1$-morphism, given by the diagonal correspondence: 
\[
X\times\overline{X}\leftarrow X\rightarrow \pt_{(n)}\quad\textrm{and}\quad
\pt_{(n)}\leftarrow X\rightarrow \overline{X}\times X\,.
\]
\item All $1$-morphisms have adjoints, and the left and right adjoints of a lagrangian correspondence ${X\leftarrow L\rightarrow Y}$ are both the original lagrangian 
correspondence that we read the other way: $Y\leftarrow L\rightarrow X$. 
The unit of the adjunction is the lagrangian correspondence $X\leftarrow L \rightarrow L\underset{Y}{\times} L$. 
\end{itemize}
We see in particular that, for a $1$-endomorphism $X\leftarrow L\rightarrow X$, its trace is 
\[
\Tr L = L\underset{\overline{X}\times X}{\times}X\,.
\]

\begin{proof}[Proof of \Cref{thm:magic-cube}]
The previous formula applies in particular to get
\begin{align*}
X_{123} 
&\simeq \Tr {L_{10}\circ L_{02}\circ L_{20}\circ L_{03}\circ L_{30} \circ L_{01}}\\
&\simeq \Tr {L_{01}\circ L_{10}\circ L_{02}\circ L_{20}\circ L_{03}\circ L_{30}} \,.
\end{align*}
where the second equivalence comes from the cyclicity of the trace.

For every $i=1,2,3$, we have a unit $2$-morphism
\[
X_0\longleftarrow  L_{0i}\longrightarrow L_{0i}\underset{X_i}{\times} L_{i0}\,.
\]
Composing these (horizontally) we get a $2$-morphism
\[
X_0\longleftarrow L_{123} \longrightarrow 
L_{01}\underset{X_1}{\times} L_{10}\underset{X_0}{\times}L_{02}\underset{X_2}{\times} L_{20}
\underset{X_0}{\times}L_{03}\underset{X_3}{\times} L_{30}\,.
\]
Now, recall that, in a symmetric monoidal bicategory with duals and adjoints, 
if $f:x\to x$ is a $1$-endomorphism and $\alpha:\Id_x\Rightarrow f$ is a $2$-morphism, we have a $1$-morphism $\bfone\to\Tr f$ (where $\bfone$ is the monoidal unit).
Let us explicit the construction in terms of correspondences.
We fix a 1-endomorphism $1\leftarrow L \to \overline{X}\times X$
and a 2-morphism $L\leftarrow M \to X$. 
The trace of $L$ fits into the diagram
\[
\begin{tikzcd}[sep=small]
&& M \ar[rdd,bend left]\ar[ldd,bend right] \ar[d,dashed]\\
&& \Tr L \ar[rd] \ar[ld]\\
& L \ar[rd]\ar[ld]&& X\ar[rd]\ar[ld] \\
1 && \overline{X} \times X && 1 \,.
\end{tikzcd}
\]
where the middle square is cartesian.
The 1-morphism $\bfone\to\Tr L$ is the lagrangian morphism $M \to \Tr L$.
We use this to get a $1$-morphism
\[
\pt_{(n)}\longleftarrow L_{123} \longrightarrow X_{123}\,.
\]
This is the lagrangian morphism we were looking for.
\end{proof}

\subsection{Shifted moment maps and derived symplectic reduction}\label{sec-momentmaps}

Let $G$ be an algebraic group, and let $\B G$ be its classifying stack. 
Then the shifted cotangent stack $\coT[n]\B G$ can be identified with the (shifted) coadjoint quotient stack $\g^*[n-1]/G$, where $\g$ is the Lie algebra of the group $G$. 
Through this identification, the tautological $1$-form of degree $n$ is 
\[
\lambda=\sum_ix_i\otimes \xi^i\,,
\] 
where the $x_i$ are a basis of $\g$ (and thus a choice of linear coordinates on $\g^*$), and $\xi^i$ are the dual basis. 
Hence 
\[
\omega=\sum_i (d_{dR}x_i) \otimes \xi^i\,.
\]
Recall the two lagrangian morphisms $\B G\to \coT[n]\B G=\g^*[n-1]/G$ and $\g^*[n-1]\to \g^*[n-1]/G$. 
The first one is the zero section, and the second one is the conormal morphism $\coN[n](\pt\to \B G)$. 
\begin{Def}\label{def:momentmap}
Let $X$ be an $n$-shifted symplectic derived stack together with a $G$-action and a $G$-equivariant morphism $\mu:X\to \g^*[n]$. 
An \textit{$n$-shifted moment map structure} on $\mu$ is a lagrangian structure on $\mu\red: X/G\to \g^*[n]/G$ together with an equivalence 
\begin{equation}
\label{eqn:reduction1}
X/G\underset{\g^*[n]/G}{\times}\g^*[n]\simeq X
\end{equation}
of $n$-shifted symplectic derived stacks. We define the \textit{derived symplectic reduction} of $X$ as the lagrangian intersection 
\begin{equation}
\label{eqn:reduction2}
X\red:= X/G\underset{\g^*[n]/G}{\times}\B G\,.
\end{equation}
\end{Def}

We refer to \cite{Cal15,SafCS} for why ordinary symplectic reduction fits into this framework. 
\begin{Rem}\label{remark-Zmu}
This is a particular instance of the magic cube (notice the faces corresponding to equations \eqref{eqn:reduction1} and \eqref{eqn:reduction2}, and that the vertical maps are all quotients by some action of $G$):
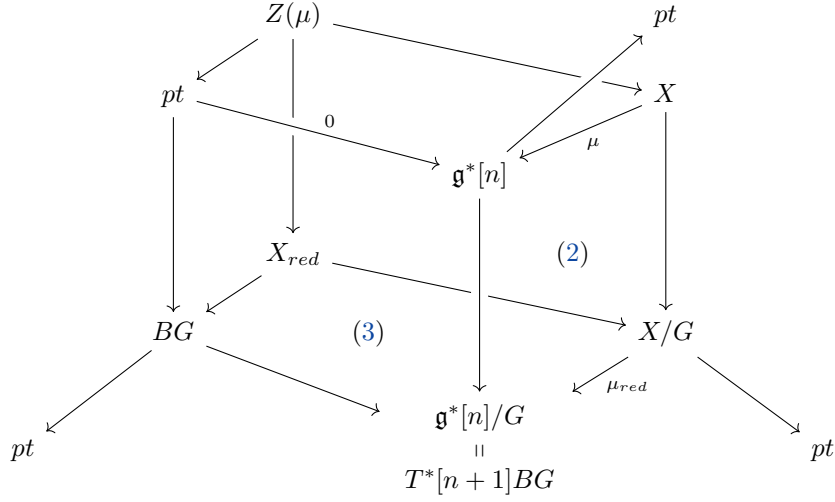
\[
\begin{tikzcd}[sep=scriptsize]
&&& Z(\mu) \ar[ld]\ar[rrd]\ar[ddd] && \pt\\
&&\pt \ar[ddd]		&&& X\ar[ld,"\mu"]\ar[ddd]\\
&&&& \g^*[n] \ar[from=llu, crossing over,"0" ]\ar[uur, crossing over] \ar[rdd,phantom, "\eqref{eqn:reduction1}"]\\
&&& X\red \ar[ld]\ar[rrd]\ar[rdd,phantom, "\eqref{eqn:reduction2}"]\\
&&\B G \ar[rrd]\ar[lld] &&&	{ X/G} \ar[ld,"\mu\red"]\ar[rrd]\\
\pt&&&& 
\begin{array}{c} \g^*[n]/G \\ \shortparallel\\ \coT[n+1]\B G\end{array} \ar[from=uuu,crossing over]&&& \pt \,.
\end{tikzcd}
\]
In particular, $L_{123}$ can be seen to be the derived zero locus of the moment map $Z(\mu):=X\underset{\g^*[n-1]}{\times}\pt$, and that 
$X_{123}\simeq X\red\times \overline{X}$. Hence \Cref{thm:magic-cube} ensures that $Z(\mu)$ is a lagrangian correspondence between $X$ and $X_{red}$. 
%\sout{where $X_i=\pt$ (see \cite{BB}) for $i\neq0$ and $X_0=\coT[n]\B G$. 
%In his case the $L_{0i}\to X_0$'s are lagrangian morphisms: we set $L_{01}=\g^*[n-1]$, $L_{02}=\B G$, and $L_{03}= X/G$. 
%Then $X_{123}\simeq X\red\times \overline{X}$, and $L_{123}\simeq Z(\mu):=X\underset{\g^*[n-1]}{\times}\pt$ 
%thus gives a lagrangian correspondence between $X\red$ and $X$.}
\end{Rem}

\begin{Exa}\label{example-reducedpoint}
If $X=\pt_{(n)}$, then the zero map $\pt\to \g^*[n]$ is an $n$-shifted moment map and 
\[
(\pt_{(n)})_{red}=\B G\underset{\g^*[n]/G}{\times}\B G\simeq \g^*[n-1]/G\,.
\]
\end{Exa}

\begin{Exa}\label{example-reduction}
Let $M$ be a derived Artin stack acted on by $G$: $M$ is the pullback of the point of $\B G$ along $p: M/G\to \B G$. 
Then its shifted cotangent stack $X=\coT[n]M$ inherits a $G$-action: the stack $(\coT[n]M)/G\simeq \coT_{M/G}[n+1]\B G$ is the 
stack of sections of the relative cotangent complex $\mathbb{L}_p[n]$. 
The morphism $\mathbb{L}_p\to p^*\mathbb{L}_{\B G}[1]$ thus induces a morphism $(\coT[n]M)/G\to \coT[n+1]\B G\simeq \g^*[n]/G$. 
This morphism is $\coN[n+1]p$ (see \cref{example-conormal}), and as such it carries a lagrangian structure. 
It defines an $n$-shifted moment map for $\coT[n]M$ with derived symplectic reduction being $\coT[n] (M/G)$. 
Indeed, according to \cite[Remark 2.20]{Cal17} we have the following equivalences of $n$-shifted symplectic derived stacks: 
\[
\coT_{ M/G}[n+1]\B G\underset{\coT[n+1]\B G}{\times}\coT_{\pt}[n+1]\B G\simeq \coT[n]M
\]
and
\[
\coT_{ M/G}[n+1]\B G\underset{\coT[n+1]\B G}{\times}\coT_{\B G}[n+1]\B G\simeq \coT[n](M/G)\,.
\]
\end{Exa}
\begin{Rem}\label{remark-reduction}
In the situation of \Cref{example-reduction}, the lagrangian correspondence 
\[
\coT[n]M\leftarrow Z(\mu)\rightarrow \coT[n] (M/G)
\]
from \Cref{remark-Zmu} is equivalent to the lagrangian correspondence $\coT[n]\big(\overline{f}:M\to M/G\big)$, that is 
\[
\coT[n]M\leftarrow \overline{f}^*\coT[n] (M/G)\rightarrow \coT[n] (M/G)\,.
\]
Indeed, \Cref{example-reduction} is actually a special case of a more general situation. 
Consider a cartesian square 
\[
\begin{tikzcd}
	M \ar[r,"\overline{f}"] \ar[d,"\overline{p}"'] & X \ar[d,"p"] \\
	P \ar[r,"f"]				& B
\end{tikzcd}
\]
of derived Artin stacks (in the example $P=\pt$ and $B=\B G$). 
In this situation, we have the following magic cube: 
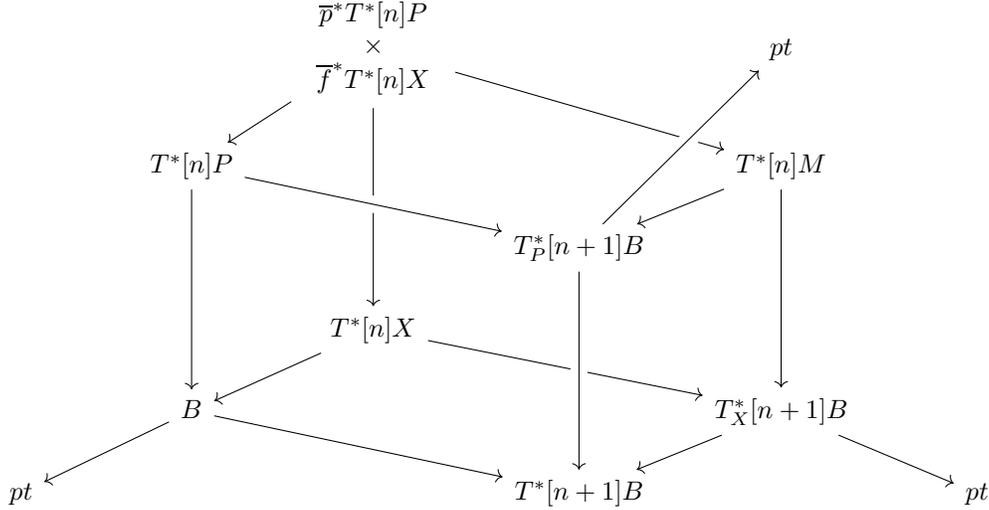
\[
\begin{tikzcd}[sep=scriptsize]
&&& \begin{array}{c}\overline{p}^*\coT[n]P\\ \times\\ \overline{f}^*\coT[n]X\end{array}\ar[ld]\ar[rrd]\ar[ddd] && \pt\\
&&\coT[n]P \ar[ddd]		&&& \coT[n]M\ar[ld]\ar[ddd]\\
&&&& \coT_P[n+1]B \ar[from=llu, crossing over]\ar[uur, crossing over]\\
&&& \coT[n]X \ar[ld]\ar[rrd]\\
&&B \ar[rrd]\ar[lld] &&&	\coT_X[n+1]B \ar[ld]\ar[rrd]\\
\pt&&&& \coT[n+1]B \ar[from=uuu,crossing over]&&& \pt\,.
\end{tikzcd}
\]
In the example ($P=\pt$, $B=\B G$, and $X=M/G$), the top corner of the cube (which is defined to be $Z(\mu)$ in \Cref{remark-Zmu}) 
becomes $\overline{f}^*\coT[n] (M/G)$. 
\end{Rem}

\subsection{Derived symplectic reduction of lagrangian morphisms}

\begin{Def}\label{def:lag-reduction}
Let $X\to\g^*[n]$ be an $n$-shifted moment map as in \Cref{def:momentmap}, and $L\to X$ be a lagrangian morphism. 
Recall the lagrangian correspondence $X\red\leftarrow Z(\mu)\rightarrow X$ from \Cref{remark-Zmu}. 
A \textit{derived symplectic reduction} of $L$ along $\mu$ is a lagrangian morphism $L\red\to X\red$ together with an equivalence 
\[
L\red\underset{X\red}{\times}Z(\mu)\simeq L
\] 
of lagrangian morphisms with codomain $X$. 
\end{Def}
\begin{Rem}\label{rem:lag-reduction}
Equivalently, we have a composition of lagrangian correspondences
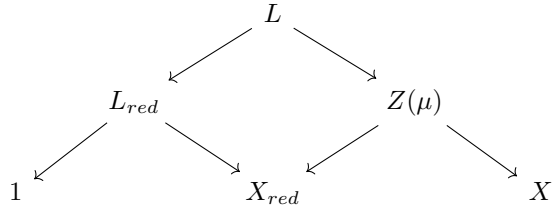
\[
\begin{tikzcd}
&&L\ar[rd]\ar[ld]\\
&L\red \ar[rd]\ar[ld] && Z(\mu)\ar[rd]\ar[ld]\\
1&&X\red &&X
\end{tikzcd}
\]
(where the square is cartesian).
Recall that $Z(\mu)\to X\red$ is the quotient of the $G$-action on $Z(\mu)$.
This implies that the map $L\to L\red$ is also the quotient map of a $G$-action on $L$.
In particular, we see that if $L$ admits a derived symplectic reduction, it inherits a $G$-action such that $L\red\simeq  L/G$.
Moreover the map $L\to X$ is $G$-equivariant and factors through $Z(\mu)$. 
\end{Rem}

\begin{Exa}\label{example-lagrangianreduction}
Keeping the notation from the previous subsection, let $\alpha$ be a degree $n$ closed $1$-form on $X= M/G$, and consider $\beta=g^*\alpha$ (where $g$ is the quotient morphism 
$M\to  M/G$). 
Then \Cref{lem-graph-pullbacks} tells us that $\Graph(\alpha)\simeq \Graph(\beta)/G$ is a derived symplectic reduction of $\Graph(\beta)$. 
\end{Exa}
\begin{Exa}\label{example-reduction-as-lagrangian-reduction}
Recall from \Cref{example-reducedpoint} that $(\pt_{(n)})_{red}\simeq \g^*[n-1]/G$. Hence a derived symplectic reduction of a lagrangian $L\to \pt_{(n)}$ 
(that is, an $(n-1)$-shifted symplectic $L$) corresponds to the data of an $(n-1)$-shifted moment map $L\to\g^*[n-1]$, with $L_{red}\simeq L/G\to\g^*[n-1]/G$. 
\end{Exa}

%%%%%%%%%%%%% SECTION 3 %%%%%%%%%%%%%%%%

\section{Relative derived critical loci}\label{section3}

\subsection{Relative zeroes of closed $1$-forms}
\label{sec:relative-zeroes}

\begin{Def}[Relative derived critical locus]
Let $p:X\to B$ be a morphism of derived Artin stacks and $\alpha$ be a closed $1$-form of degree $n$ on $X$.
The lagrangian morphism of \textit{zeroes of $\alpha$ relative to $p$} is defined as 
\[
Z_p(\alpha):=p_\dagger\big(\Graph(\alpha)\big)\,.
\]
When $\alpha=d_{dR}f$, we name it the \textit{relative (or, constrained) derived critical locus} of $f$ along $p$, and write 
\[
\Crit_p(f):=Z_p(d_{dR}f)\,.
\]
\end{Def}
\begin{Rem}
The above definition is inspired by \cite{RS}, where the authors work in the differentiable context and assume that $p$ is a surjective submersion between manifolds. 
Of course, in derived geometry, such an assumption is not necessary anymore.
\end{Rem}
Let now $X\overset{p}{\longrightarrow}Y\overset{q}{\longrightarrow}B$ be a sequence of derived Artin stacks. 
%\notemat{changer pour $M\xrightarrow {\overline f} X\xrightarrow p B$ (en cohérence avec le carré plus bas) ?}
\begin{Lem}\label{lem:compo-zeroes}
There is an equivalence $q_\dagger\big(Z_p(\alpha)\big)\simeq Z_{q\circ p}(\alpha)$, in $\Lag{\coT[n]B}$. 
\end{Lem}
\begin{proof}
This is a direct consequence of associativity of the composition of lagrangian correspondences: 
\[
q_\dagger\big(Z_p(\alpha)\big)=q_\dagger p_\dagger \Graph(\alpha)\simeq (q\circ p)_\dagger \Graph(\alpha)=Z_{q\circ p}(\alpha)\,.
\]
\end{proof}
In diagrammatic terms, the above proof reads as 
\[
\begin{tikzcd}[sep=small]
		&										&														& Z_{q\circ p}(\alpha)\ar[dl]\ar[dr]	& 								&							&			\\
		&										& Z_p(\alpha)\ar[dl]\ar[dr]	& 													& p^*q^*\coT[n]B\ar[dl]\ar[dr] 	&							&	\\
		& X\ar[dr,"\alpha"]\ar[dl]	& 													& p^*\coT[n]Y\ar[dl]\ar[dr]	&									& q^*\coT[n]B\ar[dl]\ar[dr]	&	\\
		\pt	&								& \coT[n]X 									& 													& \coT[n]Y 				& 						& \coT[n]B \,.
    \end{tikzcd}
\]

\begin{Exa}\label{usefulexa}
In particular, if we apply the above Lemma to the sequence $X\overset{p}{\to}B\overset{q}{\to} \pt$, then we get that the $(n-1)$-shifted symplectic stack $Z(\alpha)$ is equivalent to $q_\dagger Z_p(\alpha)$, which is the derived lagrangian intersection of $Z_p(\alpha)$ with the zero section in $\coT[n]B$. 
% HERE IS A DIAGRAM THAT REPRESENTS THIS
%\[
%  \xymatrix{
%					&														& Z(\alpha) \ar[dl] \ar[dr]			& 															&										\\
%					& X \ar[dl]_{p} \ar[dr]^0	& 															& Z_p(\alpha) \ar[dl] \ar[dr] &										\\
%B \ar[dr]	& 													& p^*\coT[n]B \ar[dl] \ar[dr]	&																& Y \ar[dl]_{\alpha}\\
%					& \coT[n]B 									& 															& \coT[n]X 											& 	
%  }
%\]
Moreover, if $B=\B G$ is the classifying stack of an algebraic group $G$ we get: 
a lagrangian morphism $Z_p(\alpha) \to \coT[n]\B G\simeq \g^*[n-1]/G$,
such that $Z(\alpha)$ is the lagrangian intersection of $Z_p(\alpha)$ with $\B G \to \g^*[n-1]/G$. 
In other words, $Z(\alpha)$ can be obtained as an $(n-1)$-shifted symplectic reduction. 

In this situation, we can also define a cartesian square
\[
\begin{tikzcd}
	M \ar[r,"\overline{f}"] \ar[d,"\overline{p}"']& X \ar[d,"p"] \\
	\pt \ar[r,"f"']				& \B G\,.
\end{tikzcd}
\]
Let $\beta$ be the pullback of $\alpha$ along $\overline f : M\to X$.
We will see in \Cref{usefulcor} that $Z(\beta)=Z_{\overline{p}}(\beta)$ is closely related to $Z_p(\alpha)$; this will allow us to prove (\Cref{thm-main}) 
that $Z(\alpha)$ is a derived symplectic reduction of $Z(\beta)$. 
\end{Exa}

\subsection{Pullbacks}\label{ssec-pullbacks}

We abstract the situation of \Cref{usefulexa} by considering 
a cartesian square of derived Artin stacks 
\[
\begin{tikzcd}
	M \ar[r,"\overline{f}"] \ar[d,"\overline{p}"'] \ar[rd,"h" description] & X \ar[d,"p"] \\
	P \ar[r,"f"']				& B    \,.
\end{tikzcd}
\]
Recall from \Cref{sssec-cotcon}, the associated square
\[
\begin{tikzcd}
	\Lag {\coT[n]M} \ar[from=r,"\overline{f}^\dagger"'] \ar[d,"\overline{p}_\dagger"'] & \Lag{\coT[n]X} \ar[d,"p_\dagger"] \\
	\Lag {\coT[n]P} \ar[from=r,"f^\dagger"']				& \Lag {\coT[n]B} \,.
\end{tikzcd}
\]
The following proposition says that this new square commutes.

\begin{Prop}[Beck--Chevalley condition]\label{prop-BC}
There is an equivalence $f^\dagger p_\dagger\cong \overline{p}_\dagger (\overline{f})^\dagger$ in the space of maps from $\Lag{\coT[n]X}$ to $\Lag{\coT[n]P}$. 
\end{Prop}
\begin{proof}
We have the following commuting diagram of derived stacks 
\[
\begin{tikzcd}[sep=small]
&&\coT[n]M \ar[dd, phantom, "(a)"]&& \\
&\overline{f}^*(\coT[n]X)\ar[ld]\ar[ru]\ar[dd, phantom, "(b)"]
&&\overline{p}^*(\coT[n]P)\ar[rd]\ar[lu]\ar[dd, phantom, "(c)"]& \\
\coT[n]X&&
h^*(\coT[n]B)
\ar[ru]\ar[lu]\ar[rd]\ar[ld]\ar[dd, phantom, "(d)"]&&\coT[n]P \\
&p^*(\coT[n]B)\ar[rd]\ar[lu]&&f^*(\coT[n]B)\ar[ld]\ar[ru]& \\
&&\coT[n]B&&
\end{tikzcd}
\]
which actually holds over $\Ap 1 n$. 
The squares $(a)$ and $(d)$ are cartesian because the original square is.
%The square $(b)$ is cartesian because $h^\dagger = (f')^\dagger p^\dagger$.
%The square $(c)$ is cartesian because $h^\dagger = (p')^\dagger f^\dagger$.
As a consequence we get the equivalences 
\[
p^\dagger f^*(\coT[n]B)\simeq h^*(\coT[n]B)\simeq \overline{f}_\dagger\overline{p}^*(\coT[n]P)
\]
in $\Map \bfLag {\coT[n]X} {\coT[n]P}$. 
Therefore, for every $L\in\Lag{\coT[n]X}$ we have canonical equivalences in $\Lag{\coT[n]P}$: 
\begin{align*}
f^\dagger p_\dagger L 
%& = f^*(\coT[n]B)\underset{\coT[n]B}{\times}p_\dagger L	\\
%& = (p_\dagger L)\underset{\coT[n]B}{\times}f^\dagger(\coT[n]B) \\
& = \Big(L\underset{\coT[n]X}{\times}p^*(\coT[n]B)\Big)\underset{\coT[n]B}{\times}f^*(\coT[n]B) \\
& \simeq L\underset{\coT[n]X}{\times}h^*(\coT[n]B)
&& \left(\begin{array}{l}\text{using that $(d)$ is cartesian}\end{array}\right)\\
& \simeq \Big(L\underset{\coT[n]X}{\times}\overline{f}^*(\coT[n]X)\Big)\underset{\coT[n]M}{\times}\overline{p}^*(\coT[n]P)
&& \left(\begin{array}{l}\text{using the commutativity of $(b)$ and $(c)$}\\ \text{and that $(a)$ is cartesian}\end{array}\right)\\
%\\
%& \simeq L\underset{\coT[n]X}{\times}(p')^\dagger f^\dagger(\coT[n]B) \\
%& \simeq L\underset{\coT[n]X}{\times}f'_\dagger p^\dagger\coT[n]P \\
%& \simeq f'_\dagger\pi^\dagger\coT[n]P\underset{\coT[n]X}{\times}L \\
%& \simeq (p')^*\coT[n]P\underset{\coT[n]X}{\times}(f')^\dagger L \\
& = \overline{p}_\dagger \overline{f}^\dagger L\,.
\end{align*}
\end{proof}
Let $\alpha$ be a $1$-form of degree $n$ on $X$, and write $\beta:=\overline{f}^*\alpha$. 
\begin{Cor}\label{usefulcor}
There is an equivalence $Z_{\overline{p}}(\beta)\simeq f^\dagger Z_p(\alpha)$, in $\Lag{\coT[n]P}$. 
\end{Cor}
\begin{proof}
We apply \Cref{prop-BC} to $\Graph(\alpha)$ to get
\[
f^\dagger p_\dagger \Graph(\alpha) \simeq  \overline{p}_\dagger \overline{f}^\dagger \Graph(\alpha)\,.
\]
The result follows using that $\overline{f}^\dagger \Graph(\alpha)\simeq \Graph(\beta)$.
\end{proof}

\begin{Rem}\label{magic-remark}
This is again an instance of the magic cube: 
\begin{equation}
\label[diagram]{diag:magic}
\stepcounter{equation}
\tag*{(\theequation)}
\begin{tikzcd}[sep=small]
&&&&& \coT[n]P\\
&&& Z_{\overline{p}}(\beta) \underset{\coT[n]B}{\times}B \ar[ld]\ar[rrd]\ar[ddd] \\
&& P \ar[ddd]	&&& f^\dagger Z_p(\alpha) \simeq Z_{\overline{p}}(\beta) \ar[ld]\ar[ddd]\\
&&&& f^*\coT[n]B \ar[from=llu, crossing over]\ar[uuur, crossing over, bend left=5]\\
&&& B\underset{\coT[n]B}{\times}Z_p(\alpha) \simeq Z(\alpha)\ar[ld]\ar[rrd]\\
&& B \ar[rrd]\ar[lld] &&&	Z_p(\alpha) \ar[ld]\ar[rrd]\\
\pt &&&& \coT[n]B \ar[from=uuu,crossing over]&&& \pt
\end{tikzcd}
\end{equation}
Using that $Z(\beta) \simeq Z_{\overline{p}}(\beta)\underset{\coT[n]P}{\times}P$, 
we get a lagrangian morphism
\[
Z_ {\overline{p}}(\beta)\underset{\coT[n]B}{\times}B = L_{123} \longrightarrow X_{123}=Z(\alpha)\times \overline{Z(\beta)}\,,
\] 
that is a lagrangian correspondence between $Z(\alpha)$ and $Z(\beta)$.
\end{Rem}

\begin{Exa}[Lagrange multipliers]
Assume $B=\mathbb{A}^n$ and $P=\pt$. 
Hence the map $f:\pt \to B$ is equivalent to a point $a= (a_1,\dots,a_n)$ in the affine $n$-space.
If we further assume that $X$ is a smooth algebraic variety and the point $f$ is a regular value of $p:X\to \mathbb{A}^n$, then $M$ is the smooth subvariety defined by $p(x)=f$. 
In this situation, \Cref{usefulcor} allows to describe the derived zero locus 
$Z(\alpha_{|M})$ of the restriction to $M$ of a closed $1$-form $\alpha$ on $X$ as the derived lagrangian intersection 
\[
Z_p(\alpha)\underset{\coT \mathbb{A}^n}{\times}\coT_a\mathbb{A}^n\,,
\]
where $\coT_a\mathbb{A}^n$ is the fiber of $\coT\mathbb{A}^n$ at $a$.
Vectors $\lambda=(\lambda_1,\dots,\lambda_n)\in \coT_a\mathbb{A}^n \simeq \mathbb{A}^n$ are called \textit{Lagrange multipliers}, 
and $Z_p(\alpha)$ is the derived scheme of pairs $(x,\lambda)$ in $X\times \coT_a\mathbb{A}^n$ such that $\alpha_x=p^*(\lambda)_x$. 
\end{Exa}
\begin{Rem}
In the above example, we can replace $B=\mathbb{A}^n$ with an arbitrary Artin stack.
For $B=\B G$, we get $X=M/G$ and the space of Lagrange multipliers becomes $\g^*[-1]$. 
This suggests there is some shifted moment map in play, and it is what the next subsection will show. 
\end{Rem}

\subsection{Derived symplectic reduction of critical loci}

This section proves \Cref{thm:A,thm:B}, but stated under a more general form.
%
%\bigskip
We first consider the situation of \Cref{ssec-pullbacks}, specialized to the squares
\[
\begin{tikzcd}
	M \ar[r,"\overline{f}"] \ar[d,"\overline{p}"'] & X \ar[d,"p"] \\
	\pt \ar[r,"f"']				& \B G
\end{tikzcd}
\qquad\qquad
\begin{tikzcd}
	\Lag {\coT[n]M} \ar[from=r,"\overline{f}^\dagger"'] \ar[d,"\overline{p}_\dagger"'] & \Lag{\coT[n]X} \ar[d,"p_\dagger"] \\
	\Lag {\pt_{(n)}} \ar[from=r,"f^\dagger"']				& \Lag {\coT[n]\B G}\,.
\end{tikzcd}
\]
%where $f$ is the canonical base point of $\B G$.
The stack $M$ carries a $G$-action and $X\simeq M/G$. 
Let $\alpha$ be a closed 1-form of degree $n$ on $X$, and $\beta:= \overline f^*(\alpha)$, its pullback on $M$.
Recall from \cref{sec:relative-zeroes}, the lagrangian morphism of relative zeroes $Z_p(\alpha) \to \coT[n]\B G$. 
We get the following general form for \Cref{thm:A}: 

\begin{Thm}\label{thm-main}
There are equivalences of $(n-1)$-shifted symplectic derived stacks 
\[
Z_p(\alpha)\underset{\coT[n]\B G}{\times}\g^*[n-1]\simeq Z(\beta)
\]
and 
\[
Z_p(\alpha)\underset{\coT[n]\B G}{\times}\B G\simeq Z(\alpha)\,.
\]
In other words, according to \Cref{def:momentmap}, the $(n-1)$-shifted symplectic derived stack $Z(\beta)$ is equipped with a $G$-action admitting a shifted moment map 
\[
Z(\beta)\longrightarrow \g^*[n-1]\,,
\]
and its derived symplectic reduction is equivalent to $Z(\alpha)$. 
\end{Thm}
\begin{proof}
We have $Z_p(\alpha) = p_\dagger \Graph(\alpha)$ and $Z(\beta)= \overline p_\dagger \Graph(\beta)$.
The fiber product with the map $\coT[n](f) = \g^*[n-1] \to \coT[n]\B G$ is the functor $f^\dagger$.
Then, the first equivalence is an application of \Cref{usefulcor}.
The second equivalence has already appeared in \Cref{usefulexa}. 
It is a consequence of \Cref{lem:compo-zeroes}: the fiber product with the zero section $\B G \to \coT[n]\B G$ is the functor $g_\dagger$ for $g:\B G \to \pt$.
\end{proof}

\begin{Rem}
\Cref{thm:A} is the particular case where $n=0$, $\alpha = d_{dR}(S\red)$ and $\beta = d_{dR}S$ for some a $G$-invariant function $S:X\to \mathbb A^1$. 
In this case, $Z_p(\alpha) = \Crit_p(S\red)$, $Z(\alpha) = \Crit (S\red)$ and $Z(\beta) = \Crit(S)$.
\end{Rem}

\begin{Rem}
According to \Cref{magic-remark}, the lagrangian correspondence between $Z(\alpha)$ and $Z(\beta)$ is indeed given by 
\[
Z(\beta)\underset{\coT[n]\B G}{\times}\B G\simeq Z(\beta)\underset{\g^*[n-1]}{\times}\pt\,,
\]
that is to say the derived zero fiber of the shifted moment map. Indeed, whenever $P=\pt$ and $B=\B G$, \Cref{diag:magic} becomes 
\[
\begin{tikzcd}[sep=small]
&&&&& \pt\\
&&& \B G \underset{\coT[n]\B G}{\times}Z(\beta) \ar[ld]\ar[rrd]\ar[ddd] \\
&& \pt \ar[ddd]		&&&  Z(\beta) \ar[ld]\ar[ddd]\\
&&&& \g^*[n-1] \ar[from=llu, crossing over]\ar[uuur, crossing over]\\
&&& {\B G\underset{\coT[n]\B G}{\times}Z_p(\alpha) \simeq Z(\alpha)} \ar[ld]\ar[rrd]\\
&& \B G \ar[rrd]\ar[lld] &&&	Z_p(\alpha) \ar[ld]\ar[rrd]\\
\pt &&&& \coT[n]\B G \ar[from=uuu,crossing over]&&& \pt\,.
\end{tikzcd}
\]
\end{Rem}

\medskip

We now consider the situation of \Cref{ssec-pullbacks}, for $M\to P$ an equivariant morphism of $G$-stacks, that is for a square
\[
\begin{tikzcd}
	M \ar[r,"\overline{f}"] \ar[d,"\overline{p}"'] & M/G \ar[d,"p"] \\
	P \ar[r,"f"']				& P/G \,.
\end{tikzcd}
\]
We get the following general form for \Cref{thm:B}:
\begin{Thm}\label{thm-main2}
Let $\alpha$ be a closed $1$-form of degree $n$ on $ M/G$, and let $\beta:=\overline{f}^*\alpha$. Then the relative derived criticial locus $Z_p(\alpha)$ 
is a derived symplectic reduction of $Z_{\overline{p}}(\beta)$ along the moment map $\mu:\coT[n]P\to\g^*[n]$. 
\end{Thm}
\begin{proof}
This is a direct consequence of \Cref{usefulcor}, together with \Cref{example-reduction} and \Cref{remark-reduction}. 
\end{proof}

\begin{Rem}
\Cref{thm:B} is the particular case where $n=0$, $\alpha = d_{dR}(S\red)$ and $\beta = d_{dR}S$ for some a $G$-invariant function $S:X\to \mathbb A^1$. 
In this case, $Z_p(\alpha) = \Crit_p(S\red)$, $Z_{\overline p}(\beta) = \Crit (S)$.
\end{Rem}

%%%%%%%%%%%%% SECTION 4 %%%%%%%%%%%%%%%%

\section{Symplectic reduction commutes with lagrangian intersections}\label{section3bis}

\subsection{Proof of \Cref{thm:C}}

The goal of this section is to prove the following general form of \Cref{thm:C} (which will correspond to the case $n=0$).
We need some notation first.
Let $\mu:X\to\g^*[n]$ be an $n$-shifted moment map as in \Cref{def:momentmap}.
We consider two lagrangian morphisms $L\to X$ and $L'\to X$ together with a choice of two \textit{derived symplectic reductions} $L\red\to X\red$ and $L'\red\to X\red$ of $L$ and $L'$ along $\mu$ (\Cref{def:lag-reduction}). 
Recall from \Cref{rem:lag-reduction} 
that $L$ and $L'$ inherit actions of $G$,
and that the two morphisms $L\to X\leftarrow L'$ are $G$-equivariant.
In particular the $(n-1)$-shifted symplectic derived stack $L\underset{X}{\times}L'$ admits an action of $G$.
The following result will prove in particular that this action is hamiltonian.

\begin{Thm}\label{thm-main3}
There exists an $(n-1)$-shifted moment $L\underset{X}{\times} L'\to\g^*[n-1]$  for the $(n-1)$-shifted symplectic derived $G$-stack $L\underset{X}{\times} L'$, and an equivalence 
\[
\big(L\underset{X}{\times} L'\big)\red\simeq L\red\underset{X\red}{\times} L'\red
\]
of $(n-1)$-shifted symplectic derived stacks. 
\end{Thm} 

\begin{proof}
We first construct the moment map.
Recall from \Cref{def:lag-reduction} and \Cref{rem:lag-reduction} that the maps $L\to X$ and  $L'\to X$ factor through $Z(\mu)$.
Hence, we get a commutative diagram:
\begin{equation}
\label[diagram]{diag:nu-map}
\stepcounter{equation}
\tag*{(\theequation)}
\begin{tikzcd}
L \ar[r]\ar[d] &  X \ar[d,"\mu"] & L' \ar[l]\ar[d] \\
\pt \ar[r,"0"] & {\g^*[n]} & \pt \,. \ar[l,"0"']
\end{tikzcd}
\end{equation}
Computing the limit of the two rows, we get a map 
\[
\nu : L \underset{X}{\times} L' \longrightarrow \g^*[n-1]\,.
\]
This is our candidate for the $(n-1)$-shifted moment map.
According to \Cref{def:momentmap} we need to provide a lagrangian structure on $\big(L \underset{X}{\times} L'\big)/G\to \g^*[n-1]/G$, together with an equivalence
\[
\big(L \underset{X}{\times} L'\big)/G\underset{\g^*[n-1]/G}{\times}\g^*[n-1]
\quad \simeq \quad
L \underset{X}{\times} L'
\]
of $(n-1)$-shifted symplectic stacks.

Notice that \Cref{diag:nu-map} is entirely made of $G$-equivariant maps, so the two limits inherits a $G$-action and the map $\nu$ is $G$-equivariant.
Recall also from \Cref{rem:lag-reduction} that $L/G = L\red$ and that $L'/G = L'\red$.
Hence, \Cref{diag:nu-map} induces a diagram:
\begin{equation}
\label[diagram]{diag:nu-map:2}
\stepcounter{equation}
\tag*{(\theequation)}
\begin{tikzcd}
L/G \ar[r]\ar[d] &  X/G \ar[d,"\mu"] & L'/G \ar[l]\ar[d] \\
\B G \ar[r,"0"] & {\g^*[n]/G} & \B G \,. \ar[l,"0"']
\end{tikzcd}
\end{equation}

Let us denote by $\boxplus$ the following finite category: 
\[
\begin{tikzcd}
nw \ar[r]\ar[d] & n \ar[d] & ne \ar[l]\ar[d]\\
w \ar[r] & o & e \ar[l]  \\
sw \ar[r]\ar[u]& s\ar[u]  & se \,. \ar[u]\ar[l]
\end{tikzcd}
\]
If $D:\boxplus\to \mathcal C$ is a $\boxplus$-shaped diagram in an $(\infty,1)$-category $\mathcal C$ with finite limits, the limit of $D$ can be computed in two different ways:
\begin{itemize}
\item First computing three vertical pullbacks, and then the remaining horizontal one; 
\item First computing three horizontal pullbacks, and then the remaining vertical one.
\end{itemize}

We consider the following morphism $u$ in $\mathbf{dSt}^{\boxplus}$:
\[
\begin{tikzcd}
L/G \ar[r]\ar[d] 
&  X/G \ar[d,"\mu/G"'] 
& L'/G \ar[l]\ar[d] 
&&
& \pt \ar[r, equal]\ar[d, equal] 
& \pt \ar[d,"0"'] 
& \pt \ar[l, equal]\ar[d, equal]
\\
\B G \ar[r] 
& \g^*[n]/G 
& \B G \ar[l]  
&\ar[r,"u"] &{}
& \pt \ar[r,"0"] 
& \Apcl 2 {n+1}
& \pt \ar[l,"0"']
\\
\pt \ar[r,"0"]\ar[u]
& \g^*[n] \ar[u]  
& \pt \ar[u]\ar[l,"0"'] 
&&
& \pt \ar[r, equal]\ar[u, equal]
& \pt \ar[u,"0"]  
& \pt \,, \ar[u, equal]\ar[l, equal]
\end{tikzcd}
\]
where the top of left diagram is \Cref{diag:nu-map:2}.
The natural transformation $u$ is essentially induced by the symplectic structure $\g^*[n]/G \to \Apcl 2 {n+1}$, and the naturality comes form the lagrangian structures of the four maps into $\g^*[n]/G$ of the left diagram.

The vertical limit gives a diagram
\[
\begin{tikzcd}
L \ar[r]\ar[d] 
&  X \ar[d] 
& L' \ar[l]\ar[d] 
\\
\pt \ar[r,"0"]
& \Apcl 2 n 
& \pt \,,\ar[l,"0"']
\end{tikzcd}
\]
where the vertical maps are induced by the natural transformation $u$.
Because the lagrangian map $X/G \to \g^*[n]/G$ is part of a moment map structure for $X$, the middle vertical map gives back the $n$-shifted symplectic structure of $X$.
Taking the horizontal limit the map 
$L \underset{X}{\times} L' \longrightarrow \Apcl 2 {n-1}$
is the $(n-1)$-shifted symplectic structure of $L \underset{X}{\times} L'$ as a lagrangian intersection.

Computing first the horizontal limit of $u$ we get a diagram
\[
\begin{tikzcd}
L/G \underset{X/G}{\times} L'/G \ar[r] \ar[d]
& \pt \ar[d,"0"] \\
\g^*[n-1]/G \ar[r]
& \Apcl 2 n\\
\g^*[n-1] \ar[r] \ar[u]
& \pt\,, \ar[u,"0"']
\end{tikzcd}
\]
where the horizontal maps are induced by $u$.
Using that $L/G \underset{X/G}{\times} L/G \simeq \big(L\underset{X}{\times}L\big)/G$, the top square gives a lagrangian structure on the map $\nu/G:\big(L\underset{X}{\times}L\big)/G \to \g^*[n-1]/G$.
The vertical limit of the first column is then a lagrangian intersection
whose $(n-1)$-shifted symplectic structure is given by the map resulting from the vertical limit:
\[
\big(L \underset{X}{\times} L'\big)/G\underset{\g^*[n-1]/G}{\times}\g^*[n-1] \longrightarrow \Apcl 2 {n-1}.
\]
Using that the two ways to compute the limit of $u$ coincide, we get the expected equivalence of $(n-1)$-shifted symplectic stacks:
\[
\big(L \underset{X}{\times} L'\big)/G\underset{\g^*[n-1]/G}{\times}\g^*[n-1]
\quad \simeq \quad
L \underset{X}{\times} L'\,.
\]
This finishes to prove that $\nu:L\underset{X}{\times}L \to \g^*[n-1]$ is an $(n-1)$-shifted moment map.

\medskip
We now prove the equivalence $\big(L\underset{X}{\times} L'\big)\red\simeq L\red\underset{X\red}{\times} L'\red$.
We consider the following morphism $v$ in $\mathbf{dSt}^{\boxplus}$:
\[
\begin{tikzcd}
L/G \ar[r]\ar[d]
&  X/G \ar[d]
& L'/G \ar[l]\ar[d] 
&&
& \pt \ar[r, equal]\ar[d, equal] 
& \pt \ar[d,"0"'] 
& \pt \ar[l, equal]\ar[d, equal]
\\
\B G \ar[r] 
& {\g^*[n]/G}
& \B G \ar[l]  
&\ar[r] &{}
& \pt \ar[r,"0"] 
& {\Apcl 2{n+1}} 
& \pt \ar[l,"0"']
\\
\B G \ar[r, equal]\ar[u, equal]
& \B G \ar[u]  
& \B G \ar[u, equal]\ar[l, equal] 
&&
& \pt \ar[r, equal]\ar[u, equal]
& \pt \ar[u,"0"]  
& \pt \,,\ar[u, equal]\ar[l, equal]
\end{tikzcd}
\]
where the natural transformation $v$ is essentially the symplectic structure $\g^*[n]/G \to \Apcl 2 {n+1}$, 
and where the naturality comes form the lagrangian structures of the four maps into $\g^*[n]/G$ of the left diagram.

Recall that $L\red = L/G$ and $L'\red = L'/G$.
Computing the limit first vertically, we get
\[
\begin{tikzcd}
L\red \ar[r]\ar[d] 
&  X\red \ar[d] 
& L'\red \ar[l]\ar[d] 
\\
\pt \ar[r,"0"]
& \Apcl 2 n 
& \pt\,, \ar[l,"0"']
\end{tikzcd}
\]
where the vertical maps are induced by $v$.
The middle vertical map is the $n$-shifted symplectic structure of $X\red$ and the two squares are the lagrangian structures of $L\red\to X\red$ and $L'\red\to X\red$.
%and where $X\red$ is the stack of zero of the $n$-shifted  moment map $\mu:X\to \g^*[n]$.
The horizontal limit then gives the $(n-1)$-shifted symplectic structure of 
\[
L\red\underset{X\red}{\times} L'\red
\]
coming from the lagrangian intersection.

Computing now the limit first horizontally, we get
\[
\begin{tikzcd}
L/G \underset{X/G}{\times} L/G \ar[r] \ar[d,"\nu/G"']
& \pt \ar[d,"0"] \\
\g^*[n-1]/G \ar[r]
& \Apcl 2 n\\
\B G \ar[r] \ar[u]
& \pt \,,\ar[u,"0"']
\end{tikzcd}
\]
where the top square is the lagrangian structure of the map $\nu/G$.
Using again $L/G \underset{X/G}{\times} L/G \simeq \big(L \underset{X}{\times} L\big)/G$, the vertical limit is the $(n-1)$-shifted symplectic structure
\[
\big(L \underset{X}{\times} L\big)\red \longrightarrow \Apcl 2 {n-1}.
\]
Using that the two ways to compute the limit of $v$ coincide, we get the expected equivalence of $(n-1)$-shifted symplectic stacks:
\[
\big(L\underset{X}{\times} L'\big)\red\simeq L\red\underset{X\red}{\times} L'\red\,.
\]
\end{proof}

\subsection{\Cref{thm:A} as a special case of \Cref{thm:C}}

Let $M$ be a derived $G$-stack, and assume we are given a closed $1$-form $\alpha$ of degree $n$ on the quotient $M/G$. 
We denote by $\beta$ the pullback of $\alpha$ along the quotient map $M\to M/G$.
and consider $L$ to be the lagrangian morphism $M\simeq\Graph(\beta)\to \coT[n]M$.
Recall that there is an $n$-shifted moment map $\mu:\coT[n]M\to \g^*[n]$, that $(\coT[n]M)\red\simeq \coT[n](M/G)$, and that 
$L\red:=\Graph(\alpha)$ defines a derived symplectic reduction of $L=\Graph(\beta)$. 
When $\alpha=0$, we get that the zero section $L'\red$ of $\coT[n](M/G)$ is a derived symplectic reduction of the zero section $L'$ of $\coT[n]M$. 
\Cref{thm-main3} tells us that $L\underset{\coT[n]M}{\times} L'\simeq Z(\beta)$ carries an $(n-1)$-shifted moment map, and that its 
derived symplectic reduction is $L\red\underset{\coT[n](M/G)}{\times} L'\red\simeq Z(\alpha)$.

\section{Examples}\label{section4}

\subsection{Twisted cotangent bundles of global quotient stacks}

Let $X$ be a smooth scheme, let $G$ be an affine group scheme acting on $X$, and let $\beta\in H^1(X,\Omega^{1,cl}_X)$ be a degree $1$ closed $1$-form on $X$ 
(see \Cref{ex-twisted}): $Z(\beta)$ is equivalent to the twisted cotangent bundle $\coT_{\beta}X$. 
Let us further assume that $\beta$ lifts to a degree $1$ closed $1$-form $\alpha$ on the quotient stack $X/G$. \Cref{thm-main} says that the $G$-action 
on $\coT_{\beta}X$ is hamiltonian and that $(\coT_{\beta}X)_{red}\simeq \coT_{\alpha}(X/G):=Z(\alpha)$. 

\begin{Rem}
According to the description of the underlying graded complexes from \Cref{example:smoothscheme,example:globalquotient}, the morphism 
$\DR {X/G} ^\sharp\to \DR X ^\sharp$ consists in projecting onto $C^0(G,-)$ and sending $\g^*$ to $0$. 
Hence a $1$-form $\beta_0\in H^1(X,\Omega^1_X)$ of degree $1$ on $X$ lifts to a $1$-form $\alpha_0$ on $ X/G$ if and only if it is $G$-invariant 
and basic (i.e.~if $\iota_{\vec{x}}\beta_0=0$ for every $x\in\g$). 
Moreover, as the closedness of $\beta_0$ turns out to be a property, we can see that, 
for degree reason, a degre $1$ closed $1$-form $\beta\in H^1(X,\Omega^{1,cl}_X)$ lifts to a degree $1$ closed $1$-form $\alpha$ on $ X/G$ if and only if the 
underlying $1$-form is $G$-invariant and basic. 
\end{Rem}
Lifts are not unique: we can for instance modify any given lift $\alpha$ by adding an infinitesimal character (see \Cref{exa-classifying}). 
Indeed, for every $\lambda:\B G\to \Apcl 1 1$, we have a new lift
\[
\alpha_{\lambda}=\alpha+p^*\lambda\,,
\]
of $\beta$, where $p: X/G\to \B G$. Even when $\beta=0$ we recover something interesting (and known). 
The twisted cotangent stack $\coT_{\chi} (X/G)$, for $\chi=p^*\lambda$, can be identified with the symplectic reduction of $\coT X$ along the moment map 
\[
\coT X\longrightarrow \g^*
\]
sending a covector $(x,\xi)$ to the linear map $\g \to k$ defined by  
\[
v\longmapsto \xi(\vec{v}_x)-\chi(v)\,.
\]
We refer to \cite{Grat} for more details about this example. 

\subsection{Hilbert scheme of $\mathbb{C}^3$}

We fix a positive integer $n$, and we let $M$ be the space of $4$-tuples $(x,y,z,v)\in \mathfrak{gl}_n(\mathbb{C})^{\times 3}\times\mathbb{C}^n$ 
such that $v$ is a cyclic vector for $(x,y,z)$; this means that $\mathbb{C}\langle x,y,z\rangle\cdot v=\mathbb{C}^n$. 
We finally consider the function $S=\Tr {[x,y]z}$, which is invariant under the action of $GL_n(\mathbb{C})$ on $M$ (by conjugation on $x,y,z$, 
and via the regular representation on $v$). This action does not have non-trivial stabilizers, so that $M/GL_n(\mathbb{C})$ is a smooth quasi-projective variety, 
known as the \textit{noncommutative Hilbert scheme} $\mathrm{NCHilb}^n(\mathbb{C}^3)$. 

\medskip

The truncation of the derived critical locus $\Crit(S\red)$ is the sub-scheme of $\mathrm{NCHilb}^n(\mathbb{C}^3)$ cut-out by $[x,y]=[y,z]=[z,x]=0$. 
This is isomorphic to the \textit{(commutative) Hilbert scheme} $\mathrm{Hilb}^n(\mathbb{C}^3)$. The fact that $\mathrm{Hilb}^n(\mathbb{C}^3)$ is 
the truncation of a $(-1)$-shifted symplectic derived scheme implies in particular that it carries a symmetric obstruction theory in the sense of \cite{BF} (see \cite{PTVV}), 
which is at the heart of the theory of Donaldson--Thomas (DT) invariants. 

\medskip

It would be interesting to know if the fact that $\Crit(S\red)$ is a derived symplectic reduction of $\Crit(S)$ (this is \Cref{thm:A}) provides any useful information 
on this symmetric obstruction theory (and on subsequent DT invariants). This is a general question that is not specific to this example, or even to critical loci: 
what are the consequences for the corresponding symmetric obstruction theory (and associated DT invariants) when a $(-1)$-shifted derived stack is obtained as a 
derived symplectic reduction along a $(-1)$-shifted moment map? 

\begin{Rem}
The above example is a special case of stacks of representations of quivers with potential. 
Indeed, many examples of quiver moduli shall fit into our framework: we refer to \cite{BCS} for a more general perspective, as well as for 
examples related to \Cref{thm:B} (see in particular \cite[Section 4.2.2]{BCS}). 
\end{Rem}

%\damienline{Voici le cadre general (je ne pense pas qu'il faille le mettre car il y a plein de detail avec des histoires de conditions de stabilite). 
%One let $Q=(V,E)$ be a quiver and $\vec{n}=(n_v)_{v\in V}$ be a dimension vector. 
%Let $M=Rep(Q,\vec{n})$ be the (linear) space of all representations of $V$ of dimension vector $\vec{n}$. 
%We let $G=GL_{\vec{n}}$ (or a subgroup of it). 
%Let $A_Q$ be the path algebra of $Q$. A (super)potential is an element $F\in A_Q/[A_Q,A_Q]$. 
%One can define its derivatives in some sense, and introduce the jacobian algebra $J(F):=A_Q/\partial F$. 
%The superpotential $F$ also induces a $G$-invariant function $f$ on $M$. 
%We then have: (the truncation of) $\Crit(f)$ is the variety of representations of $J(F)$ of dimension vector $\vec{n}$, and $\Crit(f\red)$ is the moduli of such representations. 
%The case of $Hilb(\mathbb{C}^3)$ is when $Q$ has two vertices $a,b$, vertex $a$ has $3$ loops $x,y,z$, there is one arrow $b\to a$, and no arrow $a\to b$. 
%The dimension vector is $(n,1)$, symmetry group is $GL_n=PGL_{(n,1)}$, superpotential is $x[y,z]$. 
%}

\subsection{Classical Chern--Simons theory}

Strictly speaking, the following example does not exactly fit into our framework, as it falls into the realm of derived differential geometry 
(rather than derived algebraic geometry). 
We refer to \cite{ThesePelle} for the foundations of shifted symplectic structures in the differentiable context. 
There are other technicalities occurring, because there are non Artin stacks appearing in the middle of the process. This is not so much of a problem as 
\begin{itemize}
\item all our results still work if we replace shifted symplectic (resp.~lagrangian) structures with shifted pre-symplectic (resp.~isotropic) structures; 
\item at the end, the non-degeneracy \textit{property} can be checked in an ad hoc manner. 
\end{itemize}

Let $M$ be a closed oriented $3$-dimensional manifold, and consider the space $X$ of connections on the trivial $G$-bundle on $M$: $X=\Omega^1(M,\g)$. 
The gauge group $\mathcal{G}=C^\infty(M,G)$, with Lie algebra $C^\infty(M,\g)$, acts on $X$ by $g\cdot A=gAg^{-1}-dgg^{-1}$. 
The cotangent to $X$ is $\Omega^1(M,\g)\oplus \Omega^2(M,\g)$, and the derived critical locus of the Chern--Simons functional 
\[
S(A)=\int_M\Tr {dA\wedge A+\frac23A^{\wedge 3}}
\]
is thus the fiber of the curvature map $\Omega^1(M,\g)\to \Omega^2(M,\g)$. 
The shifted moment map is given by the map 
\[
\Omega^1(M,\g)\oplus \Omega^2(M,\g)[-1]\longrightarrow \Omega^3(M,\g)[-1]
\]
that sends a connection $\nabla$ with a $2$-form $R$ to the covariant derivative $\nabla R$. 

Note that, even though the Chern--Simons functional is not necessarily gauge invariant, its derivative is, so that $d_{dR}S$ induces a closed $1$-form $\alpha_{CS}$ 
on $X/\mathcal{G}$. Hence (a variant of) \Cref{thm-main} ensures that the derived (differentiable) moduli stack $Z(\alpha_{CS})$ of flat connections (on the trivial 
$G$-bundle on $M$) is a derived symplectic reduction of the (infinite dimensional) derived (differentiable) space $\Crit(S)$ of flat connections. 

\begin{Rem}
\label{rem:infinite-dim}
In the above, we have simplified the picture, especially for what concerns duality issues in the infinite-dimensional setting. 
Note nevertheless that this does not change the output of the derived symplectic reduction (we refer to \cite{ThesePelle} for more details about this). 
\end{Rem}

\subsection{Einstein's covariance principle after Souriau}

This example also belongs to derived differential geometry \cite{ThesePelle}.
Let $M$ be a 4-manifold and $\mathrm{Met}(M)$ the space of lorentzian metrics $g$ on $M$.
We consider $\mathrm{Diff}(M)$ the (diffeological) group of diffeomorphisms of $M$, the (differentiable) stack $\B {\mathrm{Diff}(M)}$ classifies $M$-bundles (bundles whose fibers are isomorphic to $M$).
The group $G = \mathrm{Diff}(M)$ acts on $X=\mathrm{Met}(M)$ and the quotient stack $X/G$ classifies $M$-bundles whose fibers are each equipped with a lorentzian metric.
The Einstein--Hilbert action $S(g)=\int_M R\, vol$ (where $R$ is the scalar curvature of the metric $g$ and $vol$ the riemanniann density of $g$) is a map $S:\mathrm{Met}(M)\to \mathbb R$ which is equivariant for the action of $\mathrm{Diff}(M)$ on $\mathrm{Met}(M)$.
The critical locus $\Crit_ p(S)$ classifies Einstein metrics (solutions to Einstein equation in vacuum).
More generally, if a differential form $\ell$ on $\mathrm{Met}(M)$ is given (whose physical interpretation is the distribution of matter on $M$ \cite{Souriau,Sternberg}) the intersection of $d_{dR}S$ and $\ell$ are the metrics satisfying Einstein equation in the presence of matter.
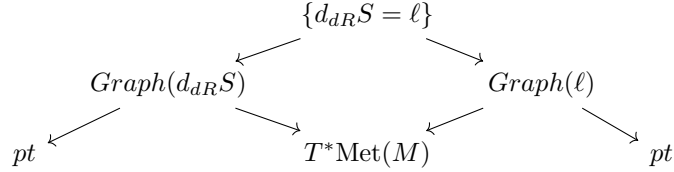
\[
\begin{tikzcd}[sep=small]
&& \{d_{dR}S = \ell\}\ar[rd]\ar[ld]	\\
& \Graph(d_{dR}S)\ar[rd]\ar[ld]&&\Graph(\ell) \ar[rd]\ar[ld]	\\
\pt && \coT \mathrm{Met}(M)&& \pt
\end{tikzcd}
\]

Souriau formulates Einstein's covariance principle as the condition of equivariance for the distribution of matter $\ell$.
If $\ell$ is produced by a distribution of symmetric tensors on a worldline of $M$, this condition implies that the worldline is geodesic, see \cite{Souriau,Sternberg}.
In our setting, this means that the distribution of matter is a 1-form on $\mathrm{Met}(M)/\mathrm{Diff}(M)$ and we are in a situation where we can apply \Cref{thm:C}.

%This stack comes with a map $p:\mathrm{Met}(M)/\mathrm{Diff}(M) \to \B{\mathrm{Diff}(M)}$ which send a point to the underlying $M$-bundle.
%The relative critical locus $\Crit_ p(S)$ is lagragian over $\coT \B{\mathrm{Diff}(M)}$ and classifies Einstein metrics on $M$-bundles relatively to the choice of the $M$-bundle.
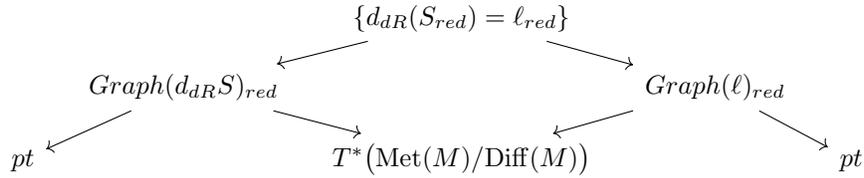
\[
\begin{tikzcd}[sep=small]
&& \{d_{dR}(S\red) = \ell\red\}\ar[rd]\ar[ld]	\\
& \Graph(d_{dR}S)\red\ar[rd]\ar[ld]&&\Graph(\ell)\red \ar[rd]\ar[ld]	\\
\pt && \coT\big(\mathrm{Met}(M)/\mathrm{Diff}(M)\big)&& \pt
\end{tikzcd}
\]

We are going to compute the shifted moment map on $\{dS = \ell\}$ explicitly but only at the infinitesimal level.
Let $\Theta(M)$ be the space of vector fields on $M$ and $Sym^2(\Theta)(M)$ the space of symmetric covariant tensors of order 2 on $M$.
The tangent space at a point $g$ of $\mathrm{Met}(M)$ is the space $Sym^2(\Omega^1)(M)$ of symmetric contravariant tensors of order 2.
The space $Sym^2(\Theta)(M)$ is a dense open subspace in the dual space of $Sym^2(\Omega^1)(M)$ and to avoid dualizability issues, we are going to consider the open subspace of $T^*_c\mathrm{Met}(M) \subset T^*\mathrm{Met}(M)$ whose fiber over a metric $g$ is the space $Sym^2(\Theta)(M)$.
We implicitely restrict $\Graph (d_{dR}S)$ and $\Graph(\ell)$ to this subspace, in what follows.
The subspace $T^*_c\mathrm{Met}(M)$ inherits the canonical 2-form of $T^*\mathrm{Met}(M)$. 
The tangent space at a point $g$ of $T^*_c\mathrm{Met}(M)$ is then the space $Sym^2(\Theta)(M) \oplus Sym^2(\Omega^1)(M)$ on which the canonical 2-form induces the obvious pairing.
We deduce that the tangent complex at a point $g$ of the stack $\{dS = \ell\}$ is 
\[
Sym^2(\Theta)(M)\longrightarrow Sym^2(\Omega^1)(M)\,,
\]
concentrated in degrees 0 and 1, and where the differential is given by the linearization of Einstein's equations at $g$.
The action of $\mathrm{Diff}(M)$ on $\mathrm{Met}(M)$ is given infinitesimally by the map 
$\mathcal L_{(-)}g:T(M) \to Sym^2(\Theta)(M)$ sending a vector field $\xi$ to the Lie derivative $\mathcal L_\xi g$.
The corresponding (infinitesimal) moment map is given by the map $Sym^2(\Omega^1)(M)\to \Omega^1(M)$ dual to $\mathcal L_{(-)}g$.

Then the induced infinitesimal action of $\mathrm{Diff}(M)$ on $\{dS = \ell\}$ is essentially given by the map $\mathcal L_{(-)}g$ in degree 0 and the derivative $\nu'$ of the $(-1)$-shifted moment map $\nu:\{dS = \ell\} \to \Omega^1(X)[-1]$ is essentially given by the map $\mu$ in degree 1.
This is best pictured as the following morphisms of complexes in amplitude $[0,1]$:
\[
\begin{tikzcd}
\Theta(M) \ar[d] && \Theta(M) \ar[d,"\mathcal L_{(-)}g"']\ar[r]& 0 \ar[d] \\
T_g\{dS = \ell\} \ar[d,"\nu'"']   && Sym^2(\Theta)(M)\ar[r]\ar[d] & Sym^2(\Omega^1)(M) \ar[d,"\mu"]\\
\Omega^1(M)[-1] && 0 \ar[r]& \Omega^1(M) \, .
\end{tikzcd}
\]
The tangent complex at $g$ in $\{dS = \ell\}\red$ is then the total complex of this double complex (in amplitude $[-1,2]$)
\[
\Theta(M) \xrightarrow {\mathcal L_{(-)}g} Sym^2(\Theta)(M)\xrightarrow {\textrm{lin. E. eqt}}  Sym^2(\Omega^1)(M) \xrightarrow {\ \mu\ }\Omega^1(M)\,.
\]

\newcommand{\bysame}{\leavevmode\hbox to3em{\hrulefill}\,}

\end{document}